\documentclass[10pt]{amsart}
\usepackage[T1]{fontenc}
\usepackage[utf8]{inputenc}
\usepackage[english]{babel}
\usepackage{csquotes} 
\usepackage[a4paper,margin=3cm]{geometry}
\allowdisplaybreaks
\usepackage{graphicx}
\usepackage[backend=biber,style=alphabetic,bibencoding=utf8,language=auto,autolang=other,giveninits=true,doi=false,isbn=false,url=false,maxnames=10,sorting=nty]{biblatex}
\usepackage{caption} 
\usepackage[hyperfootnotes=false]{hyperref} 
\hypersetup{
	colorlinks = true,
	linkcolor = {blue},
	urlcolor = {red},
	citecolor = {blue}
}
\usepackage[nameinlink,capitalise,sort]{cleveref}
\crefname{equation}{}{} 
\crefname{enumi}{}{} 

\crefname{figure}{Figure}{Figures}

\usepackage{amsmath}
\usepackage{amsthm}
\usepackage{amssymb}
\usepackage{esint} 
\usepackage{mathrsfs} 
\usepackage{faktor} 
\usepackage{dsfont} 
\usepackage[dvipsnames]{xcolor}
\usepackage{array}
\usepackage{hhline}
\usepackage[normalem]{ulem}
\usepackage{enumitem} 
\usepackage{comment} 
\usepackage[toc,page]{appendix} 
\usepackage{imakeidx} 
\usepackage{xparse} 
\usepackage{mathtools}
\usepackage{fancyhdr} 
\usepackage{ifthen} 
\usepackage{forloop} 
\usepackage{xstring}
\usepackage{tikz}
\usetikzlibrary{positioning, arrows.meta}
\usetikzlibrary{babel} 
\usetikzlibrary{cd} 
\usepackage{tikz-cd} 
\usepackage{emptypage} 
\usepackage{listings} 
\usepackage{mathabx} 
\usepackage{subcaption}
\usepackage{multirow}

\theoremstyle{plain}
\newtheorem{lemma}{Lemma}[section]
\newtheorem{proposition}[lemma]{Proposition}
\newtheorem{theorem}[lemma]{Theorem}
\newtheorem{corollary}[lemma]{Corollary}

\newtheorem{remark}[lemma]{Remark}

\theoremstyle{definition}
\newtheorem{definition}[lemma]{Definition}

\theoremstyle{remark}

\numberwithin{equation}{section}

\newcommand{\pt}{\partial_t}

\newcommand{\R}{\mathbb{R}}

\begin{document}

\title[Controllability debonding]{Controllability of a one-dimensional dynamic debonding model} 

\subjclass[2020]{%
35L05, 
93B05, 
93C20, 
35Q74, 
35R35, 
74H20.
}

\keywords{Boundary controllability; dynamic debonding; wave equation; wave equation in time-dependent domains; dynamic fracture.}

\author[N.~De Nitti]{Nicola De Nitti}
\address[N.~De Nitti]{Università di Pisa, Dipartimento di Matematica, Largo Bruno Pontecorvo 5, 56127 Pisa, Italy.}
\email[]{nicola.denitti@unipi.it}

\author[A.~Shao]{Arick Shao}
\address[A.~Shao]{Queen Mary University of London, School of Mathematical Sciences, London E1 4NS, United Kingdom.}
\email[]{a.shao@qmul.ac.uk}

\begin{abstract}
We investigate a one-dimensional dynamic debonding model, introduced by Freund (1990), in which the wave equation is coupled with a Griffith criterion governing the propagation of the fracture.
In particular, we study the boundary controllability of the system to a prescribed target state.
Our main results provide precise characterizations of the reachable target states, in both $C^{ 0, 1 }$ and $C^1$ regularity settings, and construct exact controls toward these target states.
\end{abstract}

\maketitle

\section{Introduction}
\label{sec:intro}

We study an elastodynamic model for a thin film, initially bonded to a rigid substrate and progressively detached under external loading; see \cite[Section 7.4]{MR1054375}.
The detached region deforms according to the equations of elastodynamics, and its oscillations affect the motion of the interface between the detached part and the portion still fixed to the substrate. 

Following the presentation in \cite{MR3542962}, we describe the detached region as a time-dependent, expanding domain, where the transverse displacement satisfies the linear wave equation and the evolution of this domain is itself unknown and is determined by the physical principle that the body's internal energy (kinetic and potential) remains stable. The resulting model is 
\begin{align} \label{eq.debonding_y}
\begin{cases}
\partial_{tt}^2 y (t, x) - \partial_{xx}^2 y (t, x) = 0 \text{,} & t > 0 \text{, } 0 < x < \ell (t) \text{,} \\
y (t, 0) = u(t) \text{,} & t \geq 0 \text{,} \\
y (t, \ell(t)) = 0 \text{,} & t \geq 0 \text{,} \\
y (0, x) = y_0 (x) \text{,} & 0 \leq x \leq \ell_0 \text{,} \\
\pt y (0, x) = y_1 (x) \text{,} & 0 \leq x \leq \ell_0 \text{,}
\end{cases}
\end{align}
where the \emph{displacement} $y$ is governed by the linear wave equation.
Moreover, the evolution of the \emph{debonding front} $\ell$---which acts as a free boundary for the domain of $y$---is determined by a set of conditions involving the internal energy known as \emph{Griffith's criterion}:
\begin{align} \label{eq:griff}
\begin{cases}
\ell' (t) \geq 0 \text{,} & t > 0 \text{,} \\
G_{ \ell' (t) } (t) \leq \kappa ( \ell(t) ) \text{,} & t > 0 \text{,} \\
\ell' (t) \big[ G_{ \ell' (t) } (t) - \kappa ( \ell(t) ) \big] = 0 \text{,} & t > 0 \text{,} \\
\ell(0) = \ell_0 \text{.}
\end{cases}
\end{align}
In the above, $\ell_0 > 0$ is a given initial debonding front,
\begin{equation}
\label{eq:griff_G} G_{ \ell' (t) } (t) = \tfrac{1}{2} \big( 1 - \ell' (t)^2 \big) ( \partial_x y (t, \ell(t)) )^2 \text{,} \qquad t > 0
\end{equation}
represents the \emph{dynamic energy release rate at speed $\ell' (t)$}, while $\kappa: [0, +\infty) \rightarrow [ c_1, c_2 ]$, for some fixed $0 < c_1 < c_2$, is a locally Lipschitz function representing the \emph{local toughness} of the glue between the film and the substrate.
This kind of model has recently attracted much attention; see, for instance, \cite{MR3542962,zbMATH07146759,zbMATH07761251,zbMATH07200119,zbMATH07194619,zbMATH07407723,zbMATH07713410,zbMATH06845747}.

In this article, we are interested in the following \emph{exact controllability problem} for \crefrange{eq.debonding_y}{eq:griff_G}: For $T>0$ (sufficiently large) and a target state $( \bar{\ell}_0, \bar{y}_0, \bar{y}_1 )$, find a boundary control $u$ such that 
\[
\ell (T) = \bar{\ell}_0 \text{,} \qquad y( T, \cdot ) = \bar{y}_0 \text{,} \qquad \partial_t y (T, \cdot ) = \bar{y}_1 \text{.}
\]
Implicit in this question is the issue of determining the \textit{minimal controllability time}, as well as determining which target states are reachable.

\subsection{Control of Wave Equations}

For wave equations, a fundamental obstruction to controllability is the finite speed of propagation.
Indeed, some minimum amount of time is required for information from the boundary controls to reach the entire domain.
In particular, this implies a lower bound on the controllability time (depending on the region where the control is applied) that is necessary for any boundary controllability result to hold.  

The controllability of linear wave equations on a fixed domain has been extensively studied for several decades (see, e.\,g., \cite{MR2302744,MR2549374,MR3220862,MR4823719} and references therein).
Pioneering contributions are due to Russell \cite{MR274917,MR294896}.
A systematic approach is provided by the \textit{Hilbert Uniqueness Method} (HUM) of Lions \cite{MR838402,MR931277}, which is closely connected to the abstract functional analytic framework developed in \cite{MR451141}. 

In one spatial dimension, exact controllability results can be obtained directly from d'Alembert's formula \cite[Chapter~0, pp.~1--6]{MR1359765}.
Another approach is based on Fourier series, making use of Ingham's inequality \cite{MR1545625} and its generalizations applied to the Fourier expansions of the adjoint solution.
These techniques also recover the optimal controllability time.

For general dimensions, multiplier methods provide another class of proofs (see \cite{MR838598,MR1359765,MR1388836,MR1871454}); however, such approaches usually fail to yield sharp lower bounds for the controllability time.
Optimal results were obtained via microlocal analysis, most notably by Bardos, Lebeau, and Rauch \cite{MR1178650}, who showed that an controllability inequality holds if and only if the \emph{geometric control condition} (GCC) is satisfied (see also \cite{MR1451210,MR1483711,MR3570496,MR3649373}).
Carleman estimates form yet another family of techniques that have been successfully applied to obtain controllability results, most importantly for wave equations with general lower-order coefficients (see, e.\,g., \cite{zbMATH06190724,zbMATH01588532,zbMATH01595987,zbMATH01519378}).

Control problems for the wave equation in moving domains have been studied since \cite{MR603085}.
There is a substantial amount of literature in one spatial dimension.
In \cite{MR3424115,MR3358470}, multiplier methods were applied to the case of boundaries consisting of two timelike lines; \cite{MR3335761} treated more general domains but without achieving the optimal time; \cite{MR3850022,MR4025991} employed Fourier methods; \cite{MR4376328} used the method of characteristics; and, more recently, \cite{MR3964826} obtained the optimal result using Carleman estimates.
Among other works on this subject, we also mention \cite{MR3348033,MR3657928,MR3747447,MR4186960,MR3029175,MR4303792}, as well as \cite{zbMATH05770530}, which considers a different type of free-boundary problem for the wave equation.
While far less literature exists in higher dimensions, there are some results using multiplier methods (see, e.\,g., \cite{MR1388836}), and more recently using Carleman inequalities (see \cite{MR4450884,MR3964826}).

\subsection{Novel Contributions and Outline} 

In this paper, we study the system \crefrange{eq.debonding_y}{eq:griff_G}, where the boundary evolution depends on the solution itself through a subtle energy criterion.
This feature places the problem outside the scope of previously established results, where the evolution of the domain is prescribed independently of the solution.
There are some recent results on controllability of free boundary problems (Stefan problem, porous medium equation, viscous Burgers with moving interface, etc.), which are closer in spirit but still quite different from our setting (see \cite{MR4253800,MR4173854,MR3048665,MR3204334,2203.03012} and references therein).

Our analysis shows that \emph{exact boundary controllability holds for solutions of \crefrange{eq.debonding_y}{eq:griff_G}}, provided the target state satisfies natural (and necessary) admissibility conditions.
In particular, we obtain a \emph{comprehensive characterization of all the reachable states}, and we achieve a \emph{minimal control time that is optimal in general}.
We also give an \emph{explicit construction of the desired boundary control $u$}.

We obtain controllability results for two classes of solutions to \crefrange{eq.debonding_y}{eq:griff_G}:
\begin{itemize}
\item \emph{$C^{ 0, 1 }$-solutions} ($\ell, y \in C^{ 0, 1 }$): These solutions represent the minimal regularity for which the well-posedness theory of \cite{MR3542962} holds for the system \crefrange{eq.debonding_y}{eq:griff_G}.

\item \emph{$C^1$-solutions} ($\ell, y \in C^1$): These solutions, while only slightly more regular, most crucially no longer allow for shocks arising from derivative discontinuities.
\end{itemize}
In particular, the additional challenge for our $C^1$-controllability result is that we must construct controls that do not contain derivative discontinuities that lead to shocks.


In \cref{sec:prelim}, we collect some preliminary results for solutions of \crefrange{eq.debonding_y}{eq:griff_G}.
In \cref{ssec:representation}, we recall the d'Alembert-type representation formula for $y$ and the evolution equation for $\ell$ derived in \cite{MR3542962}.
In \cref{ssec:wp}, we obtain well-posedness theorems for both $C^{ 0, 1 }$ and $C^1$ solutions.

In \cref{sec:admissible}, we turn to the question of which target states are achievable via boundary controls. In particular, we establish necessary conditions that any attainable target must satisfy: 
\begin{itemize}
\item \textit{Size constraints:} We show the expansion of the domain has a damping effect on the system, which leads to bounds in \cref{thm.obs_damp} that the final data of the system must satisfy.

\item \textit{Compatibility at the boundary:} For $C^1$-solutions, further compatibility conditions arise (see \cref{thm.obs_compat}), which lead to the distinction between admissible \emph{passive final state} and \emph{active final state}. On the other hand, for $C^{ 0, 1 }$-solutions, no such restriction arise. 
\end{itemize}
We also identify portions of $\ell$ that are determined by either the initial or final state:
\begin{itemize}
\item In \cref{sec:admissible_initial}, we identify the \emph{initial branch} of $\ell$ (from $t = 0$ to some $t = t_\ast > 0$) that is determined solely by the initial configuration $( y_0, y_1, \ell_0)$ and the Griffith law.

\item In \cref{sec:admissible_final}, we tackle the much trickier problem of identifying the \emph{final branch} of $\ell$ (from some time $t = \bar{t}_\ast > 0$ to $t = T$) that is determined by the target state.
\end{itemize}
The key ingredient in the latter is the notion of an \emph{admissible final branch} (see \cref{def.final_set} and \cref{def.final_branch})---a constrained differential inclusion representing backward-in-time solutions of the Griffith front that also satisfy the above constraints and compatibility.
This branch encodes the portion of $\ell$ that is forced by the target state and must be matched by any feasible control.

In \cref{sec.control}, we state and prove our main controllability results.
\begin{itemize}
\item \textit{$C^{0,1}$ case (\cref{thm.control_pre}):} \emph{If the initial and target states are $C^{ 0, 1 }$, and if the latter possesses an associated admissible final branch, then, given a large enough timespan $T$, there exists a $C^{0, 1}$ boundary control $u$ that exactly steers the system to the target state}.

\item \textit{$C^1$ case (\cref{thm.control}):} \emph{If initial and target states are in addition $C^1$, then under additional compatibility conditions at the boundary and a stronger admissible final branch condition, we can find a $C^1$ boundary control $u$ that exactly steers the system to the target state}.
\end{itemize}
Of particular interest is the \emph{special case in which the final branch of the debonding front is static} ($\ell' \equiv 0$ near $t = T$).
In this case, we can dispense entirely with the (rather technical) admissible final branch condition, as \emph{the reachability of a target state is determined solely by the aforementioned size constraints (\cref{thm.obs_damp}), compatibility conditions (\cref{thm.obs_compat}), and the timespan $T$}.
This leads to results that are far more natural to state, with easily checked assumptions; these are given in \cref{thm.control_static_pre} ($C^{0,1}$) and \cref{thm.control_static} ($C^1$).

The minimum control time achieved by our main results is essentially given by the sum of the timespan of the initial branch, the timespan of the final branch, and the minimal time required to expand $\ell$ from the end of the initial to the start of the final branch.
See \cref{rmk.control_time} for an argument that this minimal time is in fact optimal in general.

\subsection{Comments on the Proof}

Our proof bypasses classical HUM arguments and relies instead on d'Alembert's-type formulas---we construct the control explicitly via the d'Alembert-type representation formulas in \cref{ssec:representation}.
While this precludes our methods extending to models in higher dimensions, it allows us to take full advantage of the deep interplay among between propagation, fracture mechanics, and control design specified by the Griffith criterion.

The proof is largely based on the following key observations:
\begin{itemize}
\item An initial branch $\ell_i$ of $\ell$ is determined entirely by the initial state $( \ell_0, y_0, y_1 )$.

\item A portion of the target state (namely, $\bar{y}_1 - \bar{y}_0'$) is determined entirely by the control $u$ on the final part $B_{ f, 2 }$ of the (fixed) boundary.

\item Both the remaining portion of the target state ($\bar{y}_1 + \bar{y}_0'$) and a final branch $\ell_f$ of $\ell$ are determined by the control $u$ on the penultimate segment $B_{ f, 1 }$ of the boundary.
\end{itemize}
The above observations lead to natural obstructions for controllability.
Since the debonding front can only move rightward, the start of the final branch $\ell_f$ must lie to the right of the end of the initial branch $\ell_i$.
Furthermore, since there is a fundamental limit to how quickly the debonding front can move (namely, $\ell' < 1$), then we must allow for enough time to move the debonding front from $\ell_i$ to $\ell_f$.
These restrictions lead to the lower bounds on control time in our main results.

The construction of the control $u$ itself proceeds in three stages:
\begin{enumerate}[label=(\arabic*)]
\item \label{it:step1} \textit{Inflation to the final branch:} We construct the first part of $u$ (prior to $B_{ f, 1 }$ above) to guide $\ell$ from the end of the initial branch $\ell_i$ to the entry point of the final branch $\ell_f$. 

\item \label{it:step2} \textit{Following the admissible final branch:} We then construct the next part of $u$ (namely, on $B_{ f, 1 }$) to both follow the final branch $\ell_f$ of $\ell$ and to set a portion (namely, $\bar{y}_1 + \bar{y}_0'$) of the desired target state.
That both parts can be achieved simultaneously is a consequence of the admissible final branch criterion mentioned above.

\item \label{it:step3} \textit{Reaching the final state:} Finally, we construct the last bit of $u$ (on $B_{ f, 2 }$) to achieve the remaining part ($\bar{y}_1 - \bar{y}_0'$) of the boundary data.
(By finite speed of propagation, this portion of $u$ does not affect the debonding front $\ell$.)
\end{enumerate}
In the $C^1$ case, additional care must be taken so both $u'$ and $\ell'$ remain continuous, leading to additional matching and interpolation in the construction of $u$---particularly in Step \cref{it:step1} above.




\section{Preliminaries}
\label{sec:prelim}
In this section, we discuss various properties of the system \crefrange{eq.debonding_y}{eq:griff_G} and their solutions $( \ell, y )$.

\subsection{Representations of the Solution}
\label{ssec:representation}

Throughout this subsection, we assume $( \ell, y )$ is a $C^{ 0, 1 }$- or a $C^1$-solution of \crefrange{eq.debonding_y}{eq:griff_G}.
The boundary conditions in \cref{eq.debonding_y} imply
\begin{equation}
\label{eq.compat_0} y ( t, 0 ) = u (t) \text{,} \qquad y ( t, \ell (t) ) = 0 \text{,} \qquad t \geq 0 \text{,}
\end{equation}
while for $C^1$-solutions, we also have
\begin{equation}
\label{eq.compat_1} \partial_t y ( t, 0 ) = u' (t) \text{,} \qquad \partial_t y ( t, \ell (t) ) + \ell' (t) \, \partial_x y ( t, \ell (t) ) = 0 \text{,} \qquad t \geq 0 \text{.}
\end{equation}
In particular, \crefrange{eq.compat_0}{eq.compat_1} motivate the compatibility conditions in our upcoming well-posedness results (see \cref{thm.wp_pre} and \cref{thm.wp} below).

The analysis in \cite[Section 1]{MR3542962} shows that $y$ can be expressed in terms of a representative function $f: [ -\ell_0, \infty ) \rightarrow \R$, which is fully determined by the conditions
\begin{align} \label{eq:f}
\begin{cases}
f (s) = \tfrac{1}{2} \int_0^{-s} ( y_0' - y_1 ) ( \sigma ) \, \mathrm d \sigma \text{,} & s \in [ -\ell_0, 0 ] \text{,} \\
f (s) = u (s) - u (0) - \tfrac{1}{2} \int_0^s ( y_0' + y_1 ) ( \sigma ) \, \mathrm d \sigma \text{,} & s \in ( 0, \ell_0 ] \text{,} \\
f ( t + \ell (t) ) = u ( t + \ell (t) ) + f ( t - \ell (t) ) \text{,} & t \in ( 0, \infty ) \text{.}
\end{cases}
\end{align}
In particular, $y$ is expressed in terms of $u$ and $f$ via the formulas
\begin{equation}
\label{eq:y} y ( t, x ) = u ( t + x ) - f ( t + x ) + f ( t - x ) \text{,} \qquad t \geq 0 \text{,} \quad 0 \leq x \leq \ell (t). \\
\end{equation}

To convert \crefrange{eq:f}{eq:y} into more convenient forms, we define the maps
\begin{equation}
\label{eq.tau} \tau_\pm: [ 0, \infty ) \rightarrow [ \pm \ell_0, \infty ) \text{,} \qquad \tau_\pm (t) \coloneqq t \pm \ell (t) \text{,}
\end{equation}
which are strictly increasing and hence invertible, and satisfy (as a direct computation shows)
\begin{equation}
\label{eq.tau_deriv} \tau_\pm' (t) = 1 \pm \ell' (t) \text{,} \qquad ( \tau_\pm^{-1} )' (s) = \tfrac{1}{ 1 \pm \ell' ( \tau_\pm^{-1} (s) ) } \text{.}
\end{equation}
Using $\tau_\pm$, we can more conveniently rewrite \cref{eq:f} as
\begin{align}
\label{eq.f} f (s) = \begin{cases}
\tfrac{1}{2} \int_0^{-s} ( y_0' - y_1 ) ( \sigma ) \, \mathrm d \sigma \text{,} & s \in [ -\ell_0, 0 ] \text{,} \\
u (s) - u (0) - \tfrac{1}{2} \int_0^s ( y_0' + y_1 ) ( \sigma ) \, \mathrm d \sigma \text{,} & s \in ( 0, \ell_0 ] \text{,} \\
u (s) + f ( ( \tau_- \circ \tau_+^{-1} ) (s) ) \text{,} & s \in ( \ell_0, \infty ) \text{.}
\end{cases}
\end{align}
Formula \cref{eq.f} defines $f$ iteratively:
\begin{itemize}
\item the first two parts of \cref{eq.f} give $f (s)$ for the smallest values of $s$;
\item the last part of \cref{eq.f} expresses $f (s)$ for larger $s$ in terms of $u (s)$ and $f (s')$, for an input $s' < s$ at which $f$ was already defined.
\end{itemize}

One undesirable feature of \cref{eq:y} is that $y ( t, x )$ depends on $u$ at a later time.
However, we can eliminate this by rewriting \cref{eq:y}, for all $t \geq 0$ and $0 \leq x \leq \ell (t)$, in the form
\begin{equation}
\label{eq.y} y ( t, x ) = \begin{cases}
\tfrac{1}{2} \int_{ x - t }^{ x + t } ( y_0' + y_1 ) ( \sigma ) \, \mathrm d \sigma \text{,} & t + x \leq \ell_0 \text{, } t \leq x \text{,} \\
u ( t - x ) + \tfrac{1}{2} \int_{ t - x }^{ t + x } ( y_0' + y_1 ) ( \sigma ) \, \mathrm d \sigma \text{,} & t + x \leq \ell_0 \text{, } t > x \text{,} \\
f ( t - x ) - f ( ( \tau_- \circ \tau_+^{-1} ) ( t + x ) ) \text{,} & t + x > \ell_0 \text{.}
\end{cases}
\end{equation}
The first part of \cref{eq.y} follows from d'Alembert's formula; the second is similarly standard but also incorporates the boundary contribution $u$.
The last part of \cref{eq.y} is an immediate consequence of the identity \cref{eq:y} and the last part of \cref{eq.f}.

Owing to \cite[Section 2.2]{MR3542962}, we know the evolution \crefrange{eq:griff}{eq:griff_G} of the debonding front (for $C^1$-solutions) can be rewritten as the following initial value problem:
\begin{align} \label{eq.griffith}
\begin{cases}
\ell' (t) = \max \Big[ \tfrac{ 2 f' ( t - \ell (t) )^2 - \kappa ( \ell (t) ) }{ 2 f' ( t - \ell (t) )^2 + \kappa ( \ell (t) ) }, 0 \Big] \text{,} & t > 0 \text{,} \\
\ell (0) = \ell_0 \text{.}
\end{cases}
\end{align}
In particular, \cref{eq.griffith} implies $\ell$ satisfies
\begin{equation}
\label{eq.timelike} 0 \leq \ell' (t) < 1 \text{,} \qquad t > 0 \text{.}
\end{equation}
For $C^{ 0, 1 }$-solutions, the properties for $\ell' (t)$ in \crefrange{eq.griffith}{eq.timelike} hold as well, but only almost everywhere. 

\begin{remark}
From \cite{MR3542962}, we recall that \emph{in order to solve our original system \crefrange{eq.debonding_y}{eq:griff_G}, it suffices to solve instead the system \cref{eq.f}, \cref{eq.griffith} for $( \ell, f )$}.
From $( \ell, f )$, one then obtains the desired solution $( \ell, y )$ to \crefrange{eq.debonding_y}{eq:griff_G} by setting $y$ as in \cref{eq.y}.
\end{remark}

It will also be useful to obtain some derivative formulas.
Differentiating \cref{eq.f} yields
\begin{align}
\label{eq.ff} f' (s) \coloneqq \begin{cases}
\tfrac{1}{2} [ - y_0' (s) + y_1 (s) ] \text{,} & s \in [ -\ell_0, 0 ] \text{,} \\
u' (s) - \tfrac{1}{2} [ y_0' ( s ) + y_1 ( s ) ] \text{,} & s \in ( 0, \ell_0 ] \text{,} \\
u' (s) + f' ( ( \tau_- \circ \tau_+^{-1} ) (s) ) \, \tfrac{ 1 - \ell' ( \tau_+^{-1} ( s ) ) }{ 1 + \ell' ( \tau_+^{-1} ( s ) ) } \text{,} & s \in ( \ell_0, \infty ) \text{.}
\end{cases}
\end{align}
Combining \cref{eq.tau_deriv}, \cref{eq.f}, and \cref{eq.y}, we then see, for all $t + x > \ell_0$ and $0 \leq x \leq \ell (t)$, that
\begin{align}
\label{eq.yy} \partial_t y ( t, x ) &= f' ( t - x ) - f' ( ( \tau_- \circ \tau_+^{-1} ) ( t + x ) ) \, \tfrac{ 1 - \ell' ( \tau_+^{-1} ( t + x ) ) }{ 1 + \ell' ( \tau_+^{-1} ( t + x ) ) } \text{,} \\
\notag \partial_x y ( t, x ) &= - f' ( t - x ) - f' ( ( \tau_- \circ \tau_+^{-1} ) ( t + x ) ) \, \tfrac{ 1 - \ell' ( \tau_+^{-1} ( t + x ) ) }{ 1 + \ell' ( \tau_+^{-1} ( t + x ) ) } \text{.}
\end{align}
The above can be more conveniently stated in terms of characteristic derivatives:
\begin{align}
\label{eq.yyy} \tfrac{1}{2} ( \partial_t - \partial_x ) y ( t, x ) &= f' ( t - x ) \text{,} \\
\notag \tfrac{1}{2} ( \partial_t + \partial_x ) y ( t, x ) &= - f' ( ( \tau_- \circ \tau_+^{-1} ) ( t + x ) ) \, \tfrac{ 1 - \ell' ( \tau_+^{-1} ( t + x ) ) }{ 1 + \ell' ( \tau_+^{-1} ( t + x ) ) } \text{.}
\end{align}
Again, for $C^{ 0, 1 }$-solutions, the equalities in \crefrange{eq.ff}{eq.yyy} only hold a.\,e.

\begin{remark}
For any $t \geq 0$ and $0 \leq x \leq \ell (t)$, we have
\begin{equation}
\label{eq.tau_mono_1} \tau_+^{-1} ( t + x ) \leq \tau_+^{-1} ( t + \ell (t) ) \leq t \text{,}
\end{equation}
where we recalled the monotonicity of $\tau_+$.
Similarly, if $t + x > \ell_0$, then we also have
\begin{equation}
\label{eq.tau_mono_2} ( \tau_- \circ \tau_+^{-1} ) ( t + x ) \leq ( \tau_- \circ \tau_+^{-1} ) ( t + \ell (t) ) = t - \ell (t) \leq t - x \text{.}
\end{equation}
Thus, in the representation formulas for $y$ in \cref{eq.y}, \cref{eq.yy}, and \cref{eq.yyy}, we have that $y ( t, x )$ and its derivatives depend only on $u$ and $\ell$ at earlier times.
\end{remark}

\subsection{Well-Posedness}
\label{ssec:wp}

We now state the relevant existence and uniqueness results for \crefrange{eq.debonding_y}{eq:griff_G}.
First, well-posedness for $C^{ 0, 1 }$-solutions was proved in \cite[Theorem 3.5]{MR3542962}, which we restate here in the setting of a finite time interval.

\begin{theorem}[Existence and uniqueness] \label{thm.wp_pre}
Let us fix $T > 0$, assume the given data satisfies
\begin{equation}
\label{eq.wp_data_pre} ( y_0, y_1 ) \in ( C^{ 0, 1 } \times L^\infty ) [ 0, \ell_0 ] \text{,} \qquad u \in C^{ 0, 1 } [ 0, T ] \text{,}
\end{equation}
and suppose the following compatibility conditions hold:
\begin{equation}
\label{eq.wp_compat_pre} y_0 (0) = u (0) \text{,} \qquad y_0 ( \ell_0 ) = 0 \text{.}
\end{equation}
Then, there exists a unique solution $( \ell, y )$ to \crefrange{eq.debonding_y}{eq:griff_G} such that
\begin{equation}
\label{eq.wp_reg_pre} \ell \in C^{ 0, 1 } [ 0, T ] \text{,} \qquad y \in C^{ 0, 1 } ( \{ 0 \leq t \leq T \text{, } 0 \leq x \leq \ell (t) \} ) \text{,} \qquad f \in C^{ 0, 1 } [ -\ell_0, T ] \text{.}
\end{equation}
\end{theorem}

\begin{remark}
The representation formula \cref{eq.y} implies slightly more regularity for $y$ than was stated in \cref{eq.wp_reg_pre}---namely, the cross-sections $y ( t, \cdot )$ and $y ( \cdot, x )$ are $C^{ 0, 1 }$ for \emph{every} fixed $t$ and $x$.
\end{remark}

We can then adapt the well-posedness result in \cref{thm.wp_pre} to $C^1$-solutions.

\begin{theorem} \label{thm.wp}
Let us fix $T > 0$, assume the given data satisfies
\begin{equation}
\label{eq.wp_data} ( y_0, y_1 ) \in ( C^1 \times C^0 ) [ 0, \ell_0 ] \text{,} \qquad u \in C^1 [ 0, T ] \text{,}
\end{equation}
and suppose the following compatibility conditions hold:
\begin{align}
\label{eq.wp_compat} y_0 (0) = u (0) \text{,} &\qquad y_0 ( \ell_0 ) = 0 \text{,} \\
\notag y_1 (0) = u' (0) \text{,} &\qquad y_1 ( \ell_0 ) = - \max \left[ \frac{ \frac{1}{2} | y_1 ( \ell_0 ) - y_0' ( \ell_0 ) |^2 - \kappa ( \ell_0 ) }{ \tfrac{1}{2} | y_1 ( \ell_0 ) - y_0' ( \ell_0 ) |^2 + \kappa ( \ell_0 ) }, \, 0 \right] y_0' ( \ell_0 ) \text{.} 
\end{align}
Then, there exists a unique solution $( \ell, y )$ to \crefrange{eq.debonding_y}{eq:griff_G} such that
\begin{equation}
\label{eq.wp_reg} \ell \in C^1 [ 0, T ] \text{,} \qquad y \in C^1 ( \{ 0 \leq t \leq T \text{, } 0 \leq x \leq \ell (t) \} ) \text{,} \qquad f \in C^1 [ -\ell_0, T ] \text{.}
\end{equation}
\end{theorem}

\begin{proof}[Proof of \cref{thm.wp}]
\cref{thm.wp_pre} provides the solution $( \ell, y )$ to \crefrange{eq.debonding_y}{eq:griff_G}, so we need only obtain the regularity \cref{eq.wp_reg}.
For this, we apply an iteration similar to that found in \cite{MR3542962}.

To start, by the first two parts of \cref{eq.f} and \cref{eq.ff}, along with \eqref{eq.wp_compat}, we have $f \in C^1 [ -\ell_0, \ell_0 )$. Here, we stress that \cref{eq.wp_compat} implies that $f$ is $C^1$ at $0$.
From \cref{eq.griffith}, we obtain $\ell \in C^1 [ 0, \tau_-^{-1} ( \ell_0 ) )$, and it follows that $\tau_\pm \in C^1 [ 0, \tau_-^{-1} ( \ell_0 ) )$.
The last parts of \cref{eq.f} and \cref{eq.ff} then yield $f \in C^1 [ -\ell_0, ( \tau_+ \circ \tau_-^{-1} ) ( \ell_0 ) )$.
Again, we need the compatibility condition \cref{eq.wp_compat} to ensure $f$ is $C^1$ at $\ell_0$.

We now iterate as follows.
Since we extended the $C^1$-regularity of $f$, applying \eqref{eq.griffith} extends the $C^1$-regularity of $\ell$ (and hence $\tau_\pm$).
Feeding this into \eqref{eq.f} and \eqref{eq.ff}, we again further extend the $C^1$-regularity of $f$, returning to the previous step.
We then repeat this indefinitely until we reach time $T$, and we obtain $\ell \in C^1 [ 0, T ]$ and $f \in C^1 [ -\ell_0, T ]$.

The above proves the first and third parts of \cref{eq.wp_reg}.
The second part of \cref{eq.wp_reg} follows from \cref{eq.y} and the above regularity of $\ell$ and $f$.
The compatibility conditions \cref{eq.wp_compat} are essential for showing that the formula for $y$ in \cref{eq.y} is indeed $C^1$.
\end{proof}

\begin{remark}
For any $C^{ 0, 1 }$-solution $( \ell, y )$ of \crefrange{eq.debonding_y}{eq:griff_G}, the compatibility conditions \cref{eq.wp_compat_pre} follow from \cref{eq.compat_0}.
Similarly, for any $C^1$-solution $( \ell, y )$ of \crefrange{eq.debonding_y}{eq:griff_G}, the conditions \cref{eq.wp_compat} follow from \crefrange{eq.compat_0}{eq.compat_1} and \cref{eq.griffith}.
Thus, \cref{eq.wp_compat_pre} and \cref{eq.wp_compat} are necessary for the well-posedness results as stated.
\end{remark}

\section{Admissible solutions}
\label{sec:admissible}

We now derive some additional a priori properties of solutions to \crefrange{eq.debonding_y}{eq:griff_G} that lead to fundamental obstructions to controllability.
For convenience, we define the following quantities in the remainder of this section:
\begin{equation}
\label{eq.final_data} T > \ell_0 \text{,} \qquad \bar{\ell}_0 \coloneqq \ell ( T ) \geq \ell_0 \text{,} \qquad \bar{y}_0 \coloneqq y ( T, \cdot ) \text{,} \qquad \bar{y}_1 \coloneqq \partial_t y ( T, \cdot ) \text{.}
\end{equation}
Observe that \cref{eq.yyy} and \cref{eq.final_data} imply the following identities, for all $0 \leq x \leq \bar{\ell}_0$:
\begin{align}
\label{eq.final_dataa} \tfrac{1}{2} ( \bar{y}_1 - \bar{y}_0' ) (x) &= f' ( T - x ) \text{,} \\
\notag \tfrac{1}{2} ( \bar{y}_1 + \bar{y}_0' ) (x) &= - f' ( ( \tau_- \circ \tau_+^{-1} ) ( T + x ) ) \, \tfrac{ 1 - \ell' ( \tau_+^{-1} ( T + x ) ) }{ 1 + \ell' ( \tau_+^{-1} ( T + x ) ) } \text{.}
\end{align}
For $C^{ 0, 1 }$-solutions, both parts of \cref{eq.final_dataa} only hold a.\,e.

\begin{remark}
The assumption $T > \ell_0$ in \cref{eq.final_data} implies
\[
\tau_+^{-1} ( T + x ) \geq \tau_+^{-1} (T) > 0 \text{,} \qquad T - x \geq ( \tau_- \circ \tau_+^{-1} ) ( T + x ) \geq ( \tau_- \circ \tau_+^{-1} ) (T) > - \ell_0 \text{,}
\]
so the quantities in \cref{eq.final_dataa} are all well-defined.
\end{remark}

\subsection{Constraint and Compatibility Conditions}

In this subsection, we derive restrictions on $( \bar{y}_0, \bar{y}_1 )$ imposed by our system \crefrange{eq.debonding_y}{eq:griff_G}.

First, we note the expanding domain has the effect of damping part of the solution.

\begin{proposition} \label{thm.obs_damp}
The following hold (pointwise for $C^1$-solutions, a.\,e.~for $C^{ 0, 1 }$-solutions):
\begin{align}
\label{eq.obs_damp} | ( \bar{y}_1 + \bar{y}_0' ) (x) |^2 &\leq 2 \, \kappa ( \ell ( \tau_+^{-1} ( x + T ) ) ) \text{,} && x \in [ 0, \bar{\ell}_0 ] \text{,} \\
\notag | ( \bar{y}_1 + \bar{y}_0' ) ( \tau_+ (t) - T ) |^2 &\leq 2 \, \kappa ( \ell (t) ) \text{,} && t \in [ \tau_+^{-1} (T), T ] \text{.} 
\end{align}
\end{proposition}

\begin{proof}
First, observe that the transformation
\begin{equation}
\label{eql.obs_damp_0} t \coloneqq \tau_+^{-1} ( x + T ) \text{,} \qquad x = \tau_+ (t) - T \text{.}
\end{equation}
gives a one-to-one correspondence between $x \in [ 0, \bar{\ell}_0 ]$ and $t \in [ \tau_+^{-1} (T), T ]$, thus it suffices to only prove the second part of \cref{eq.obs_damp}.
For this, we now split the proof into cases.

First, suppose $2 \, f' ( \tau_- (t) )^2 \leq \kappa ( \ell (t) )$.
Then, \cref{eq.griffith} yields $\ell' (t) = 0$, so by \cref{eq.final_dataa},
\begin{align*}
| ( \bar{y}_1 + \bar{y}_0' ) ( \tau_+ (t) - T ) |^2 &= 4 | f' ( \tau_- (t) ) |^2 \\
&\leq 2 \kappa ( \ell (t) ) \text{,}
\end{align*}
where we used the assumption on $f'$ in our last step.

On the other hand, if $2 \, f' ( \tau_- (t) )^2 > \kappa ( \ell (t) )$, then \cref{eq.griffith} and \cref{eq.final_dataa} yield
\begin{align*}
( \bar{y}_1 + \bar{y}_0' ) ( \tau_+ (t) - T ) &= - 2 \, f' ( \tau_- (t) ) \, \tfrac{ 1 - \ell' (t) }{ 1 + \ell' (t) } \\
&= - 2 \, f' ( \tau_- (t) ) \, \frac{ 1 - \frac{ 2 f' ( \tau_- (t) )^2 - \kappa ( \ell (t) ) }{ 2 f' ( \tau_- (t) )^2 + \kappa ( \ell (t) ) } }{ 1 + \frac{ 2 f' ( \tau_- (t) )^2 - \kappa ( \ell (t) ) }{ 2 f' ( \tau_- (t) )^2 + \kappa ( \ell (t) ) } } \\
&= - \tfrac{ \kappa ( \ell (t) ) }{ f' ( \tau_- (t) ) } \text{.}
\end{align*}
The above, along with our assumption on $f'$, immediately yields
\begin{align*}
| ( \bar{y}_1 + \bar{y}_0' ) ( \tau_+ (t) - T ) |^2 &= \tfrac{ \kappa ( \ell (t) )^2 }{ f' ( \tau_- (t) )^2 } \\
&< 2 \, \kappa ( \ell (t) ) \text{.} \qedhere
\end{align*}
\end{proof}

Next, for $C^1$-solutions, we have additional compatibility conditions for the final state.

\begin{proposition} \label{thm.obs_compat}
Let $( \ell, y )$ be a $C^1$-solution.
Then, one of the following holds:
\begin{itemize}
\item $\ell' (T) = 0$, and the following compatibility condition holds:
\begin{equation}
\label{eq.obs_compat_0} \bar{y}_1 ( \bar{\ell}_0 ) = 0 \text{.}
\end{equation}

\item $\ell' (T) > 0$, and the following compatibility condition holds:
\begin{equation}
\label{eq.obs_compat_1} 2 \, \kappa ( \bar{\ell}_0 ) < | \bar{y}_0' ( \bar{\ell}_0 ) |^2 \text{,} \qquad \bar{y}_1 ( \bar{\ell}_0 ) = - \sqrt{ 1 - \tfrac{ 2 \kappa ( \bar{\ell}_0 ) }{ | \bar{y}_0' ( \bar{\ell}_0 ) |^2 } } \cdot \bar{y}_0' ( \bar{\ell}_0 ) \text{.}
\end{equation}
\end{itemize}
\end{proposition}

\begin{proof}
For convenience, we define the shorthands
\begin{equation}
\label{eql.obs_compat_0} L \coloneqq \ell' (T) \text{,} \qquad Y \coloneqq \tfrac{ 2 \kappa ( \bar{\ell}_0 ) }{ | \bar{y}_0' ( \bar{\ell}_0 ) |^2 } \text{.}
\end{equation}
By \cref{eq.compat_1}, \cref{eq.timelike}, and \cref{eq.final_data}, we have
\begin{equation}
\label{eql.obs_compat_1} \bar{y}_1 ( \bar{\ell}_0 ) = - L \, \bar{y}_0' ( \bar{\ell}_0 ) \text{,} \qquad 0 \leq L < 1 \text{,}
\end{equation}
thus it suffices to find all possible values of $L$.

First, if $L = 0$, then \cref{eql.obs_compat_0} immediately implies \cref{eq.obs_compat_0}, proving the first case.
Observe also that this scenario is consistent with \cref{eq.griffith}, since by \cref{eq.final_dataa},
\begin{align*}
L &= \max \Big[ \tfrac{ 2 f' ( \tau_- (T) )^2 - \kappa ( \bar{\ell}_0 ) }{ 2 f' ( \tau_- (T) )^2 + \kappa ( \bar{\ell}_0 ) }, 0 \Big] \\
&= \max \Big[ \tfrac{ \frac{1}{2} | ( \bar{y}_1 + \bar{y}_0' ) ( \bar{\ell}_0 ) |^2 ( 1 + L )^2 - \kappa ( \bar{\ell}_0 ) ( 1 - L )^2 }{ \frac{1}{2} | ( \bar{y}_1 + \bar{y}_0' ) ( \bar{\ell}_0 ) |^2 ( 1 + L )^2 + \kappa ( \bar{\ell}_0 ) ( 1 - L )^2 }, 0 \Big] \text{,}
\end{align*}
and \cref{thm.obs_damp} implies $L = 0$ indeed solves the above equation.

On the other hand, if $L > 0$, then \cref{eq.griffith}, \cref{eq.final_dataa}, and \cref{eql.obs_compat_1} yield
\begin{align*}
L &= \tfrac{ \frac{1}{2} | ( \bar{y}_1 + \bar{y}_0' ) ( \bar{\ell}_0 ) |^2 \frac{ ( 1 + L )^2 }{ ( 1 - L )^2 } - \kappa ( \bar{\ell}_0 ) }{ \frac{1}{2} | ( \bar{y}_1 + \bar{y}_0' ) ( \bar{\ell}_0 ) |^2 \frac{ ( 1 + L )^2 }{ ( 1 - L )^2 } + \kappa ( \bar{\ell}_0 ) } \\
&= \tfrac{ \frac{1}{2} | \bar{y}_0' ( \bar{\ell}_0 ) |^2 ( 1 + L )^2 - \kappa ( \bar{\ell}_0 ) }{ \frac{1}{2} | \bar{y}_0' ( \bar{\ell}_0 ) |^2 ( 1 + L )^2 + \kappa ( \bar{\ell}_0 ) } \\
&= \tfrac{ ( 1 + L )^2 - Y }{ ( 1 + L )^2 + Y } \text{.}
\end{align*}
Multiplying the above by the right-hand denominator yields the cubic equation
\begin{align*}
0 &= L [ ( 1 + L )^2 + Y ] - [ ( 1 + L )^2 - Y ] \\
&= L^3 + L^2 - ( 1 - Y ) L - ( 1 - Y ) \\
&= ( 1 + L ) [ L^2 - ( 1 - Y ) ] \text{.}
\end{align*}
Since $1 + L \neq 0$ by \cref{eql.obs_compat_1}, the above yields
\begin{equation}
\label{eql.obs_compat_2} 0 < L^2 = 1 - Y \text{,} \qquad 0 < L = \sqrt{ 1 - Y } \text{.}
\end{equation}
In particular, \cref{eql.obs_compat_2} has a solution only when $Y < 1$, which gives precisely the first part of \cref{eq.obs_compat_1}.
Combining \cref{eql.obs_compat_1} and \cref{eql.obs_compat_2} yields the second part of \cref{eq.obs_compat_1}.
\end{proof}

To summarize, any final data $( \bar{\ell}_0, \bar{y}_0, \bar{y}_1 )$ for a solution of \crefrange{eq.debonding_y}{eq:griff_G} must necessarily satisfy the bound \cref{eq.obs_damp}.
Moreover, in the setting of $C^1$-solutions:
\begin{itemize}
\item One possibility is that $\ell' (T) = 0$, and \cref{eq.obs_compat_0} holds; we call this a \emph{passive final state}.

\item For sufficiently non-small final data---i.\,e.,  satisfying the first part of \cref{eq.obs_compat_1}---it is also possible that $\ell' (T) > 0$.
In this case, $\ell' (T)$ is a unique positive value, given by the second part of \cref{eq.obs_compat_1}; we call this an \emph{active final state}.
\end{itemize}
However, for $C^{ 0, 1 }$-solutions, no such restriction of $\bar{y}_1$ and $\bar{y}_0'$ exists due to derivative discontinuities, and one does not have well-defined notions of active or passive final states.

\subsection{The Initial Branch}
\label{sec:admissible_initial}

We now identify the parts of the boundary curve $\ell$ whose dynamics are entirely determined by the initial data.
Since $f$ is independent of $u$ on $[ -\ell_0, 0 ]$, then \cref{eq.f} and \cref{eq.griffith} imply $\ell (t)$ is independent of $u$ until $\tau_- (t) > 0$.
This motivates the following:

\begin{definition} \label{def.initial_branch}
We define the quantities $( t_\ast, \ell_\ast, \ell_\ast' ) \coloneqq ( t_\ast, \ell_\ast, \ell_\ast' ) ( \ell_0, y_0, y_1 )$ by
\begin{equation}
\label{eq.initial_branch} \begin{cases}
t_\ast \coloneqq \tau_-^{-1} (0) \text{,} \\
\ell_\ast \coloneqq \ell ( t_\ast ) = t_\ast \text{,} \\
\ell_\ast' \coloneqq \ell' ( t_\ast ) \text{,} & \text{$C^1$-solutions only.}
\end{cases}
\end{equation}
\end{definition}

\begin{remark}
In practice, $( t_\ast, \ell_\ast, \ell_\ast' )$ can be obtained from $( \ell_0, y_0, y_1 )$ by solving \cref{eq.f} and \cref{eq.griffith} until the time $t_\ast$ for which $\ell ( t_\ast ) = t_\ast$, and then by setting $\ell_\ast \coloneqq t_\ast$ and $\ell_\ast' \coloneqq \ell' ( t_\ast )$.
\end{remark}

\subsection{The Final Branch}
\label{sec:admissible_final}

Next, we can similarly identify the parts of $\ell$ whose dynamics are entirely determined by the final data, however the process is considerably less straightforward.

Here, the idea is to solve \cref{eq.griffith} backwards from $t = T$, given final data $( \bar{y}_0, \bar{y}_1 )$.
More specifically, given $t \in [ \tau_+^{-1} (T), T ]$, we obtain, from \cref{eq.griffith} and \cref{eq.final_dataa},
\begin{align}
\label{eql.final_branch_0} \ell' (t) &= \max \Bigg[ \tfrac{ \frac{1}{2} | ( \bar{y}_1 + \bar{y}_0' ) ( \tau_+ (t) - T ) |^2 \frac{ ( 1 + \ell' (t) )^2 }{ ( 1 - \ell' (t) )^2 } - \kappa ( \ell (t) ) }{ \frac{1}{2} | ( \bar{y}_1 + \bar{y}_0' ) ( \tau_+ (t) - T ) |^2 \frac{ ( 1 + \ell' (t) )^2 }{ ( 1 - \ell' (t) )^2 } + \kappa ( \ell (t) ) }, \, 0 \Bigg] \\
\notag &= \max \Big[ \tfrac{ | ( \bar{y}_1 + \bar{y}_0' ) ( \tau_+ (t) - T ) |^2 ( 1 + \ell' (t) )^2 - 2 \kappa ( \ell (t) ) ( 1 - \ell' (t) )^2 }{ | ( \bar{y}_1 + \bar{y}_0' ) ( \tau_+ (t) - T ) |^2 ( 1 + \ell' (t) )^2 + 2 \kappa ( \ell (t) ) ( 1 - \ell' (t) )^2 }, \, 0 \Big]
\end{align}
Owing to \cref{thm.obs_damp}, $\ell' (t) \coloneqq 0$ solves \cref{eql.final_branch_0} at every $t \in [ \tau_+^{-1} (T), T ]$.

To find other solutions to \cref{eql.final_branch_0}, we consider the remaining case $\ell' (t) > 0$.
Setting
\begin{equation}
\label{eql.final_branch_1} L \coloneqq \ell' (t) \text{,} \qquad Y \coloneqq | ( \bar{y}_1 + \bar{y}_0' ) ( \tau_+ (t) - T ) |^2 \text{,} \qquad K \coloneqq \kappa ( \ell (t) ) \text{,}
\end{equation}
then \cref{eql.final_branch_0} reduces to the algebraic equation
\[
L = \tfrac{ Y ( 1 + L )^2 - K ( 1 - L )^2 }{ Y ( 1 + L )^2 + K ( 1 - L )^2 } \text{,} \qquad Y ( 1 + L )^2 - K ( 1 - L )^2 > 0 \text{.}
\]
Rearranging the above yields a cubic equation and a quadratic inequality for $L$:
\begin{align}
\label{eql.final_branch_2} 0 &= L [ Y ( 1 + L )^2 + K ( 1 - L )^2 ] - [ Y ( 1 + L )^2 - K ( 1 - L )^2 ] \\
\notag &= ( Y + K ) L^3 + ( Y - K ) L^2 - ( Y + K ) L - ( Y - K ) \\
\notag &= [ ( Y + K ) L + ( Y - K ) ] ( L^2 - 1 ) \text{,} \\
\notag 0 &< ( Y - K ) L^2 + 2 ( Y + K ) L + ( Y - K ) \text{.}
\end{align}
Since $L^2 - 1 \neq 0$ by \cref{eq.timelike}, then the unique solution to the equation in \cref{eql.final_branch_2} is
\begin{equation}
\label{eql.final_branch_10} L = \frac{ K - Y }{ K + Y } \in [ 0, 1 ] \text{,}
\end{equation}
while the inequality in \cref{eql.final_branch_2} then expands to
\begin{align*}
0 &< - \tfrac{ ( K - Y )^3 }{ ( K + Y )^2 } + ( K - Y ) \\
&= \tfrac{ ( K - Y ) [ ( K + Y )^2 - ( K - Y )^2 ] }{ ( K + Y )^2 } \text{,}
\end{align*}
which holds as long as $Y \neq 0$, which by \cref{eql.final_branch_10} is equivalent to $L < 1$.

Unpacking the above analysis, we see that while $\ell (t)$ depends only on the final data for every $t \in [ \tau_+^{-1} (T), T ]$, there could be zero, one, or multiple possibilities for $\ell$.
In particular, in order to produce a solution, we need a curve $\ell$ satisfying, for $t \in [ \tau_+^{-1} (T), T ]$,
\begin{align*}
\ell' (t) &\in \Big\{ 0, \, \tfrac{ 2 \kappa ( \ell (t) ) - | ( \bar{y}_1 + \bar{y}_0' ) ( t + \ell (t) - T ) |^2 }{ 2 \kappa ( \ell (t) ) + | ( \bar{y}_1 + \bar{y}_0' ) ( t + \ell (t) - T ) |^2 } \Big\} \cap [ 0, 1 ) \text{,} \\
2 \kappa ( \ell (t) ) &\geq | ( \bar{y}_1 + \bar{y}_0' ) ( t + \ell (t) - T ) |^2 \text{,}
\end{align*}
in either the $C^{ 0, 1 }$-sense or the $C^1$-sense.
In particular, the second condition above is a constraint on the size of (part of) the final data, while the first condition is a differential inclusion for $\ell'$.
All of this motivates the following definitions, which in tandem characterize the possible final data sets and the possible final branches of $\ell$.

\begin{definition} \label{def.final_set}
Given any $\bar{\ell}_0 \geq \ell_0$ and $( \bar{y}_0, \bar{y}_1 ) \in ( C^{ 0, 1 } \times L^\infty ) [ 0, \bar{\ell}_0 ]$, we say the triple $( \bar{\ell}_0, \bar{y}_0, \bar{y}_1 )$ is a $C^{ 0, 1 }$-\emph{admissible final data set} iff the following condition holds:
\begin{equation}
\label{eq.final_set_0} \bar{y}_0 ( \bar{\ell}_0 ) = 0 \text{.}
\end{equation}
If in addition $( \bar{y}_0, \bar{y}_1 ) \in ( C^1 \times C^0 ) [ 0, \bar{\ell}_0 ]$, then we call $( \bar{\ell}_0, \bar{y}_0, \bar{y}_1 )$ a $C^1$-\emph{admissible final data set} iff, in addition to \cref{eq.final_set_0}, the following condition holds:
\begin{equation}
\label{eq.final_set_1} \bar{y}_1 ( \bar{\ell}_0 ) = -\alpha \, \bar{y}_0' ( \bar{\ell}_0 ) \text{,} \qquad \alpha \in \bigg\{ 0, \sqrt{ 1 - \tfrac{ 2 \kappa ( \bar{\ell}_0 ) }{ | \bar{y}_0 ( \bar{\ell}_0 ) |^2 } } \bigg\} \cap [ 0, 1 ) \text{.}
\end{equation}
\end{definition}

\begin{definition} \label{def.final_branch}
Let $T > \ell_0$ and $0 < \bar{t}_\ast < T$.
\begin{itemize}
\item Given a $C^{ 0, 1 }$-admissible final data set $( \bar{\ell}_0, \bar{y}_0, \bar{y}_1 )$, we say that $\mathscr{L} \in C^{ 0, 1 } [ \bar{t}_\ast, T ]$ is a $C^{ 0, 1 }$-\emph{admissible final branch} iff the following hold a.\,e.,
\begin{equation}
\label{eq.final_branch_0} \begin{cases}
| ( \bar{y}_1 + \bar{y}_0' ) ( t + \mathscr{L} (t) - T ) |^2 \leq 2 \kappa ( \mathscr{L} (t) ) \text{,} & t \in [ \bar{t}_\ast, T ] \text{,} \\
\mathscr{L}' (t) \in \Big\{ 0, \, \tfrac{ 2 \kappa ( \mathscr{L} (t) ) - | ( \bar{y}_1 + \bar{y}_0' ) ( t + \mathscr{L} (t) - T ) |^2 }{ 2 \kappa ( \mathscr{L} (t) ) + | ( \bar{y}_1 + \bar{y}_0' ) ( t + \mathscr{L} (t) - T ) |^2 } \Big\} \cap [ 0, 1 ) \text{,} & t \in [ \bar{t}_\ast, T ] \text{,}
\end{cases}
\end{equation}
and the following identities hold:
\begin{equation}
\label{eq.final_branchh_0} \mathscr{L} (T) = \bar{\ell}_0 \text{,} \qquad \bar{t}_\ast + \mathscr{L} ( \bar{t}_\ast ) = T \text{.}
\end{equation}

\item Given a $C^1$-admissible final data set $( \bar{\ell}_0, \bar{y}_0, \bar{y}_1 )$, we say $\mathscr{L} \in C^1 [ \bar{t}_\ast, T ]$ is a $C^1$-\emph{admissible final branch} iff \cref{eq.final_branch_0} holds everywhere, the conditions \cref{eq.final_branchh_0} hold, and the following also holds, with $\alpha$ as in \eqref{eq.final_set_1}:
\begin{equation}
\label{eq.final_branch_1} \mathscr{L}' ( T ) = \alpha \text{.}
\end{equation}
\end{itemize}
In both cases above, we define $( \bar{t}_\ast, \bar{\ell}_\ast, \bar{\ell}_\ast' ) \coloneqq ( \bar{t}_\ast, \bar{\ell}_\ast, \bar{\ell}_\ast' ) ( \bar{\ell}_0, \bar{y}_0, \bar{y}_1 )$, with $\bar{t}_\ast$ as above, by
\begin{equation}
\label{eq.ell_asts} \begin{cases}
\bar{\ell}_\ast \coloneqq \mathscr{L} ( \bar{t}_\ast ) \text{,} \\
\bar{\ell}_\ast' \coloneqq \mathscr{L}' ( \bar{t}_\ast ) \text{,} & \text{$C^1$-admissible final branch only.}
\end{cases}
\end{equation}
\end{definition}

\begin{remark}
In practice, one obtains an admissible final branch (if it exists) as follows.
One first solves the differential inclusion from the second part of \cref{eq.final_branch_0}, with final data at $t = T$ given in the first part of \cref{eq.final_branchh_0}.
(For $C^1$-solutions, one must also require that $\mathscr{L}' (t)$ varies continuously, and \cref{eq.final_branch_1} uniquely selects the value of $\mathscr{L}' (T)$.)
The inclusion is solved backwards in $t$ until time $t_\ast$, at which the second part of \cref{eq.final_branchh_0} holds.
Lastly, the first part of \cref{eq.final_branch_0} is an additional constraint that must be satisfied when $\mathscr{L}' (t) = 0$.
\end{remark}

\begin{remark}
Observe that in the above-mentioned process, there are two key obstructions to the existence of an admissible final branch $\mathscr{L}$:
\begin{itemize}
\item The constraint $| ( \bar{y}_1 + \bar{y}_0' ) ( t + \mathscr{L} (t) - T ) |^2 \leq 2 \kappa ( \mathscr{L} (t) )$ may fail to hold if $\mathscr{L}' (t) = 0$.

\item In the $C^1$-setting, if $\mathscr{L}' > 0$, then $| \bar{y}_1 + \bar{y}_0' | \rightarrow 0$ would force $\mathscr{L}' \rightarrow 1$.
\end{itemize}
\end{remark}

The admissible final branch condition, while comprehensive, is unfortunately quite technical, since it involves solving the constrained differential inclusion \cref{eq.final_branch_0}, for which solutions may not exist or be unique.
However, for $C^{ 0, 1 }$- or passive $C^1$-admissible final data, there is an admissible final branch representing a static boundary, provided the appropriate constraint holds:

\begin{proposition} \label{thm.final_branch_static}
Let $X \in \{ C^{ 0, 1 }, C^1 \}$, let $( \bar{\ell}_0, \bar{y}_0, \bar{y}_1 )$ be an $X$-admissible final data set, let $T > \bar{\ell}_0$, and suppose that the following conditions hold:
\begin{equation}
\label{eq.final_branch_static_ass} \begin{cases}
\| \bar{y}_1 + \bar{y}_0' \|_{ L^\infty [ 0, \bar{\ell}_0 ] }^2 \leq 2 \, \kappa ( \bar{\ell}_0 ) \text{,} \\
\bar{y}_1 ( \bar{\ell}_0 ) = 0 \text{,} & \text{$X \in C^1$ only.}
\end{cases}
\end{equation}
Then, the following is an $X$-admissible final branch:
\begin{equation}
\label{eq.final_branch_static} \mathscr{L}: [ T - \bar{\ell}_0, T ] \rightarrow \R \text{,} \qquad \mathscr{L} (t) = \bar{\ell}_0 \text{.}
\end{equation}
Furthermore, the parameters \eqref{eq.ell_asts} in this setting are given by
\begin{equation}
\label{eq.ell_asts_static} \bar{t}_\ast = T - \bar{\ell}_0 \text{,} \qquad \bar{\ell}_\ast = \bar{\ell}_0 \text{,} \qquad \bar{\ell}_\ast' = 0 \text{.}
\end{equation}
\end{proposition}

\begin{proof}
Since $\mathscr{L}' \equiv 0$, it follows that $\mathscr{L}$ satisfies the second part of \cref{eq.final_branch_0} and \cref{eq.final_branchh_0}.
If $X = C^1$, then the second part of \cref{eq.final_branch_static_ass} implies $\alpha = 0$, so \cref{eq.final_branch_1} is satisfied.
The first part of \cref{eq.final_branch_0} now follows from the first part of \cref{eq.final_branch_static_ass}.
Finally, \eqref{eq.ell_asts_static} is a consequence of \eqref{eq.final_branchh_0}, \eqref{eq.ell_asts}, and \eqref{eq.final_branch_static}.
\end{proof}

We will refer to $\mathscr{L}$'s arising from \cref{thm.final_branch_static} as \emph{static final branches}.

\begin{remark}
Observe that if $\kappa$ is a constant function, $\kappa \equiv \kappa_0 > 0$, then the first (constraint) condition in \eqref{eq.final_branch_0} fully decouples from the second (differential inclusion) condition in \eqref{eq.final_branch_0}.
In this case, \eqref{eq.final_branch_0} reduces to more easily checked criteria:
\begin{equation}
\label{eq.final_branch_0a} \begin{cases}
\| \bar{y}_1 + \bar{y}_0' \|_{ L^\infty [ 0, \bar{\ell}_0 ] }^2 \leq 2 \, \kappa_0 \text{,} \\
\mathscr{L}' (t) \in \Big\{ 0, \, \tfrac{ 2 \kappa ( \mathscr{L} (t) ) - | ( \bar{y}_1 + \bar{y}_0' ) ( t + \mathscr{L} (t) - T ) |^2 }{ 2 \kappa ( \mathscr{L} (t) ) + | ( \bar{y}_1 + \bar{y}_0' ) ( t + \mathscr{L} (t) - T ) |^2 } \Big\} \cap [ 0, 1 ) \text{,} & t \in [ \bar{t}_\ast, T ] \text{,}
\end{cases}
\end{equation}
\end{remark}

\section{Controllability} 
\label{sec.control}

This section is dedicated to stating and proving our exact controllability results.
In \cref{sec.control_0}, we treat the case of controlling $C^{ 0, 1 }$-solutions with $C^{ 0, 1 }$-controls, while in \cref{sec.control_1}, we turn our attention toward controlling $C^1$-solutions with $C^1$-controls.

\subsection{\texorpdfstring{$C^{ 0, 1 }$}{C01}-Controllability} \label{sec.control_0}

We begin with the main result for $C^{ 0, 1 }$-solutions:

\begin{theorem}[$C^{ 0, 1 }$-controllability] \label{thm.control_pre}
Let us assume that the following conditions hold:
\begin{enumerate}[label=(\roman*)]
\item Initial data $\ell_0 > 0$ and $( y_0, y_1 ) \in ( C^{ 0, 1 } \times L^\infty ) [ 0, \ell_0 ]$ satisfying
\begin{equation}
\label{eq.control_compat_pre} y_0 ( \ell_0 ) = 0 \text{.}
\end{equation}

\item Final time $T > \ell_0$, and $C^{ 0, 1 }$-admissible final data set $( \bar{\ell}_0, \bar{y}_0, \bar{y}_1 )$.
\end{enumerate}
Moreover, suppose there exists a $C^{ 0, 1 }$-admissible final branch $\mathscr{L} \in C^{ 0, 1 } [ \bar{t}_\ast, T ]$ such that
\begin{equation}
\label{eq.control_time_pre} \ell_\ast \leq \bar{\ell}_\ast < \bar{t}_\ast \text{.}
\end{equation}
Then, there is a control $u \in C^{ 0, 1 } [ 0, T ]$ such that the solution $( y, \ell )$ of \crefrange{eq.debonding_y}{eq:griff_G} satisfies
\begin{equation}
\label{eq.control_pre} \ell (T) = \bar{\ell}_0 \text{,} \qquad y ( T, \cdot ) = \bar{y}_0 \text{,} \qquad \partial_t y ( T, \cdot ) = \bar{y}_1 \text{.}
\end{equation}
\end{theorem}

\begin{remark}
The first inequality of \cref{eq.control_time_pre} is necessary for controllability.
In particular, since the domain size $\ell$ is always non-decreasing, then the size of $\ell_\ast$ at the end of the initial branch must be at most the size $\bar{\ell}_\ast$ at the start of the final branch.
\end{remark}

\begin{remark}
The restriction on the control time in \cref{thm.control_pre} is implicitly contained in \cref{eq.control_time_pre}, since $\bar{t}_\ast$ is computed from $T$ and the final data.
To justify \cref{eq.control_time_pre}, we note it can be rewritten as
\begin{equation}
\label{eq.control_time_pree} \bar{\ell}_\ast - \ell_\ast < \bar{t}_\ast - t_\ast \text{.}
\end{equation}
Since the initial branch of $\ell$ (before time $t_\ast$) and the final branch of $\ell$ (after time $\bar{t}_\ast$) are determined by the initial and final data, our control must expand $\ell$ from $\ell_\ast$ at time $t_\ast$ to $\bar{\ell}_\ast$ at time $\bar{t}_\ast$.
Since $\ell' < 1$, then \cref{eq.control_time_pree} implies that we have just enough time to carry out this domain expansion.
\end{remark}

\begin{remark}
If a $C^{ 0, 1 }$-admissible final branch does not exist for an admissible final data set $( \bar{\ell}_0, \bar{y}_0, \bar{y}_1 )$, then the argument preceding \crefrange{def.final_set}{def.final_branch} shows $( \bar{\ell}_0, \bar{y}_0, \bar{y}_1 )$ cannot be reachable.
Thus, \cref{thm.control_pre} gives an optimal description of the reachable states.
\end{remark}

\begin{proof}[Proof of \cref{thm.control_pre}.]
Our assumptions and \cref{thm.wp_pre} yield the existence of the $C^{ 0, 1 }$-solution $( \ell, y )$ of \crefrange{eq.debonding_y}{eq:griff_G}, provided we construct $u \in C^{ 0, 1 } [ 0, T ]$ satisfying
\begin{equation}
\label{eql.control_pre_0} u (0) = y_0 (0) \text{.}
\end{equation}
Setting $u (0)$ as in \cref{eql.control_pre_0}, it suffices to define $u' \in L^\infty [ 0, T ]$, as $u$ is then determined by integrating from $t = 0$.
We also recall the quantities $t_\ast = \ell_\ast$ from \cref{def.initial_branch}, and we recall in particular that $\ell$ is independent of $u$, and hence already determined, on $[ 0, t_\ast ]$.

\uline{Step 1.}
Our first goal is to expand the domain to the starting size $\bar{\ell}_\ast$ of the final branch.
For this, we define (while recalling \cref{eq.initial_branch} and \cref{eq.control_time_pre}) the constant
\begin{equation}
\label{eql.control_pre_1} v \coloneqq \tfrac{ \bar{\ell}_\ast - \ell_\ast }{ \bar{t}_\ast - t_\ast } \in [ 0, 1 ) \text{,}
\end{equation}
and we now \emph{prescribe} $\ell$ on $( t_\ast, \bar{t}_\ast ]$ to be the linear function
\begin{equation}
\label{eql.control_pre_2} \ell (t) \coloneqq \ell_\ast + v ( t - t_\ast ) \text{.}
\end{equation}
In particular, the above implies
\begin{equation}
\label{eql.control_pre_3} \ell ( t_\ast ) = \ell_\ast \text{,} \qquad \ell ( \bar{t}_\ast ) = \bar{\ell}_\ast \text{,} \qquad \ell' |_{ ( t_\ast, \bar{t}_\ast ] } \equiv v \text{.}
\end{equation}
For convenience, we define $\tau_\pm$ on $[ 0, \bar{t}_\ast ]$ using \cref{eq.tau}, with the above $\ell$.

With $\ell$ as in \cref{eql.control_pre_2}, we now claim there exist $u' \in L^\infty ( 0, \tau_- ( \bar{t}_\ast ) ]$ such that \cref{eq.f} and \cref{eq.griffith} hold (the latter a.\,e.) for $f \in C^{ 0, 1 } [ -\ell_0, \tau_- ( \bar{t}_\ast ) ]$ and the above extension of $\ell \in C^{ 0, 1 } [ 0, \bar{t}_\ast ]$.

To prove the claim, we note that, for every $0 < s \leq \tau_- ( \bar{t}_\ast )$, the equation \cref{eq.griffith}, with $t \coloneqq \tau_-^{-1} (s)$ and $\ell$ as in \cref{eql.control_pre_2}, yields an algebraic equation for $f' (s)$ with:
\begin{itemize}
\item One or two possible solutions when $v > 0$.

\item Infinitely many solutions---namely, any $f' (s)$ with $2 f' (s)^2 \leq \kappa ( \ell ( \tau_-^{-1} (s) ) )$---when $v = 0$.
\end{itemize}
For simplicity, we can choose the values of $f' (s)$ so that $f'$ varies continuously for $s \in ( 0, \tau_-^{-1} ( \bar{t}_\ast ) ]$.
With $f' (s)$ chosen, then (the second and third parts of) \cref{eq.ff} gives the value of $u' (s)$ needed to achieve the above desired value of $f' (s)$.
(In particular, $f' (s)$ is defined in terms of $u' (s)$ and $f'$ at an earlier time when $s > \ell_0$, so the above can be done iteratively.)

From this, we obtain
\[
u' \in L^\infty ( 0, \tau_- ( \bar{t}_\ast ) ] \text{,} \qquad f' \in L^\infty ( 0, \tau_- ( \bar{t}_\ast ) ]
\]
satisfying \cref{eq.griffith} and \cref{eq.ff}.
Integrating the above (the latter from $f (0) = 0$) yields
\[
u \in C^{ 0, 1 } [ 0, \tau_- ( \bar{t}_\ast ) ] \text{,} \qquad f \in C^{ 0, 1 } [ -\ell_0, \tau_- ( \bar{t}_\ast ) ]
\]
satisfying \cref{eq.f} and \cref{eq.griffith}, hence proving the claim.

\uline{Step 2.}
The next step is to attain the final domain size $\bar{\ell}_0$ and impose half of the final state $( \bar{y}_0, \bar{y}_1 )$.
For this, we aim to prescribe $\ell \in C^{ 0, 1 } ( \bar{t}_\ast, T ]$ so that it matches the given final branch:
\begin{equation}
\label{eql.control_pre_10} \ell (t) \coloneqq \mathscr{L} (t) \text{,} \qquad t \in ( \bar{t}_\ast, T ] \text{.}
\end{equation}
We again extend $\tau_\pm$ to $( \bar{t}_\ast, T ]$ using \cref{eq.tau} and the above $\ell$.

Using the second and third parts of \cref{eq.ff}, we can then solve for $u' \in L^\infty ( \tau_- ( \bar{t}_\ast ), \tau_- (T) ]$ so that $f' \in L^\infty ( \tau_- ( \bar{t}_\ast ), \tau_- (T) ]$ satisfies 
\begin{equation}
\label{eql.control_pre_11} f' (s) \coloneqq - \tfrac{1}{2} ( \bar{y}_1 + \bar{y}_0' ) ( ( \tau_+ \circ \tau_-^{-1} ) (s) - T ) \, \tfrac{ 1 + \ell' ( \tau_-^{-1} (s) ) }{ 1 - \ell' ( \tau_-^{-1} (s) ) } \text{,} \qquad s \in ( \tau_- ( \bar{t}_\ast ), \tau_- (T) ] \text{,}
\end{equation}
with $\ell$ as in \cref{eql.control_pre_10}.
The expression in \cref{eql.control_pre_11} is well-defined, since the right-hand side of its equation is an $L^\infty$-function of $s$ by \crefrange{def.final_set}{def.final_branch}, and since $s \in ( \tau_- ( \bar{t}_\ast ), \bar{\tau}_- (T) ]$ implies
\[
\tau_-^{-1} (s) \in ( \bar{t}_\ast, T ] \text{,} \qquad ( \tau_+ \circ \tau_-^{-1} ) (s) \in ( T, T + \bar{\ell}_0 ] \text{,}
\]
with the second part of the above following from  \cref{def.final_branch}.

We claim \crefrange{eql.control_pre_10}{eql.control_pre_11} together satisfy \cref{eq.f} and \cref{eq.griffith}.
Observe \cref{eq.f} will be satisfied due to $f'$ being defined via \cref{eq.ff}.
For \cref{eq.griffith}, we divide into cases for (almost) every $t \in ( \bar{t}_\ast, T ]$.
First, if $\ell' (t) = 0$, then (except for a zero-measure subset) by \cref{eq.final_branch_0} and \crefrange{eql.control_pre_10}{eql.control_pre_11},
\begin{align*}
2 | f' ( t - \ell (t) ) |^2 &= \tfrac{1}{2} | ( \bar{y}_1 + \bar{y}_0' ) ( t + \ell (t) - T ) |^2 \\
&\leq | \kappa ( \ell (t) ) | \text{,}
\end{align*}
and it follows that \cref{eq.griffith} holds for this $t$.

Next, if $\ell' (t) > 0$, then \cref{eq.final_branch_0} and \cref{eql.control_pre_11} yield
\begin{align*}
\ell' (t) &= \tfrac{ 2 \kappa ( \ell (t) ) - | ( \bar{y}_1 + \bar{y}_0' ) ( t + \ell (t) - T ) |^2 }{ 2 \kappa ( \ell (t) ) + | ( \bar{y}_1 + \bar{y}_0' ) ( t + \ell (t) - T ) |^2 } \\
&= \tfrac{ \kappa ( \ell (t) ) [ 1 + \ell' (t) ]^2 - 2 | f' ( t - \ell (t) ) |^2 [ 1 - \ell' (t) ]^2 }{ \kappa ( \ell (t) ) [ 1 + \ell' (t) ]^2 + 2 | f' ( t - \ell (t) ) |^2 [ 1 - \ell' (t) ]^2 } \text{.}
\end{align*}
Rearranging the above results in the cubic equation
\[
[ 1 + \ell' (t)^2 ] \{ [ \kappa ( \ell (t) ) + 2 | f' ( t - \ell (t) ) |^2 ] \ell' (t) + [ \kappa ( \ell (t) ) - 2 | f' ( t - \ell (t) ) |^2 ] \} = 0 \text{,}
\]
for which the only real solution is given by
\[
\ell' (t) = \tfrac{ 2 | f' ( t - \ell (t) ) |^2 - \kappa ( \ell (t) ) }{ 2 | f' ( t - \ell (t) ) |^2 + \kappa ( \ell (t) ) } \text{.}
\]
Since $\ell' (t) > 0$, we once again have \cref{eq.griffith}, completing the proof of the claim.

Integrating $u'$ and $f'$ as before, we now have
\[
u \in C^{ 0, 1 } [ 0, \tau_- (T) ] \text{,} \qquad f \in C^{ 0, 1 } [ -\ell_0, \tau_- (T) ] \text{,} \qquad \ell \in C^{ 0, 1 } [ 0, T ]
\]
satisfying \cref{eq.f} and \cref{eq.griffith}.
Moreover, \cref{eq.final_branchh_0} and \cref{eql.control_pre_10} now yield the first part of \cref{eq.control_pre}.

\uline{Step 3.}
The last step is to finish prescribing $( \bar{y}_0, \bar{y}_1 )$.
For this, we extend $u' \in L^\infty ( \tau_- (T), T ]$ so that $f' \in L^\infty ( \tau_- (T), T ]$ satisfies the following identity:
\begin{equation}
\label{eql.control_pre_12} f' (s) \coloneqq \tfrac{1}{2} ( \bar{y}_1 - \bar{y}_0' ) ( T - s ) \text{,} \qquad s \in [ \tau_- (T), T ] = [ T - \bar{\ell}_0, T ] \text{.}
\end{equation}
Integrating $f'$ as before, we have $f \in C^{ 0, 1 } [ -\ell_0, T ]$ that satisfies \cref{eq.f}.

Having now a solution $( \ell, f )$ of \cref{eq.f} and \cref{eq.griffith}, we obtain from \cref{eq.y} a corresponding $C^{ 0, 1 }$-solution $( \ell, y )$ of \crefrange{eq.debonding_y}{eq:griff_G}.
Furthermore, by \cref{eq.yyy}, \cref{eql.control_pre_11}, and \cref{eql.control_pre_12}, we have
\[
( \partial_t \pm \partial_x ) y ( T, x ) = ( \bar{y}_1 \pm \bar{y}_0' ) ( x ) \text{,} \qquad x \in [ 0, \bar{\ell}_0 ] \text{.}
\]
In particular, the above implies
\begin{equation}
\label{eql.control_pre_13} \partial_t y ( T, \cdot ) = \bar{y}_1 \text{,} \qquad \partial_x y ( T, \cdot ) = \bar{y}_0' \text{,}
\end{equation}
proving the third part of \cref{eq.control_pre}.
Finally, combining \cref{eq.compat_0}, \cref{eq.final_set_0}, and the second part of \cref{eql.control_pre_13} yields $y ( T, \cdot ) = \bar{y}_0$, proving the second part of \cref{eq.control_pre} and the theorem.
\end{proof}

\begin{remark} \label{rmk.control_time}
The control time in \cref{thm.control_pre} was computed by first focusing solely on expanding the domain to the start of the given final branch, and then afterwards focusing solely on setting the boundary control to attain the final state $( \bar{y}_0, \bar{y}_1 )$.
One can argue heuristically that this control time should generally be optimal.

The above two steps are already individually achieved in optimal time (the first by taking $v$ in \cref{eql.control_pre_1} to be close to $1$, and the second by \cref{eq.y}).
Moreover, expanding the domain requires an infusion of large values of $u'$ and $f'$.
Thus, if the final data is appropriately small (i.\,e.~requiring setting small $f'$), then one cannot expand the domain and impose $( \bar{y}_0, \bar{y}_1 )$ simultaneously, leading to \cref{thm.control_pre} providng the optimal control time.
\end{remark}

For the special case of static final branches, we can replace the admissible final branch assumption with a constraint condition that is simple to check.

\begin{corollary}[$C^{ 0, 1 }$-controllability, static final branch] \label{thm.control_static_pre}
Let us assume that the following conditions hold:
\begin{enumerate}[label=(\roman*)]
\item Initial data $\ell_0 > 0$ and $( y_0, y_1 ) \in ( C^{ 0, 1 } \times L^\infty ) [ 0, \ell_0 ]$ satisfying
\begin{equation}
\label{eq.control_static_compat_pre} y_0 ( \ell_0 ) = 0 \text{.}
\end{equation}

\item Final time $T > 0$ and final data $\bar{\ell}_0 > 0$, $( \bar{y}_0, \bar{y}_1 ) \in ( C^1 \times C^0 ) [ 0, \bar{\ell}_0 ]$ satisfying
\begin{equation}
\label{eq.control_static_time_pre} \bar{\ell}_0 \geq \ell_\ast \text{,} \qquad T > 2 \bar{\ell}_0 \text{,}
\end{equation}
as well as the compatibility and constraint conditions
\begin{equation}
\label{eq.control_static_const_pre} \bar{y}_0 ( \bar{\ell}_0 ) = 0 \text{,} \qquad \| \bar{y}_1 + \bar{y}_0' \|_{ L^\infty [ 0, \bar{\ell}_0 ] }^2 \leq 2 \, \kappa ( \bar{\ell}_0 ) \text{.}
\end{equation}
\end{enumerate}
Then, there exists a control $u \in C^{ 0, 1 } [ 0, T ]$ so that the solution $( y, \ell )$ of \crefrange{eq.debonding_y}{eq:griff_G} satisfies
\begin{equation}
\label{eq.control_static_pre} \ell (T) = \bar{\ell}_0 \text{,} \qquad y ( T, \cdot ) = \bar{y}_0 \text{,} \qquad \partial_t y ( T, \cdot ) = \bar{y}_1 \text{.}
\end{equation}
\end{corollary}

\begin{proof}
By \cref{eq.control_static_const_pre}, the function $\mathscr{L}$ given by \cref{eq.final_branch_static} is a $C^{ 0, 1 }$-admissible final branch, with $\bar{t}_\ast = T - \bar{\ell}_0$ and $\bar{\ell}_\ast = \bar{\ell}_0$.
The above and \cref{eq.control_static_time_pre} imply \cref{eq.control_time_pre} holds with $\mathscr{L}$ as above.
Since all the assumptions of \cref{thm.control_pre} hold in our setting, applying \cref{thm.control_pre} results in the desired control.
\end{proof}

\begin{remark}
In particular, \cref{thm.control_static_pre} implies that in the setting of static final branches, the crucial condition for a final state $( \bar{\ell}_0, \bar{y}_0, \bar{y}_1 )$ to be reachable is precisely \cref{eq.control_static_const_pre}.
\end{remark}

\subsection{\texorpdfstring{$C^1$}{C1}-Controllability} \label{sec.control_1}

We now state and prove our main result for $C^1$-solutions:

\begin{theorem}[$C^1$-controllability] \label{thm.control}
Let us assume that the following conditions hold:
\begin{enumerate}[label=(\roman*)]
\item Initial data $\ell_0 > 0$ and $( y_0, y_1 ) \in ( C^1 \times C^0 ) [ 0, \ell_0 ]$ satisfying
\begin{equation}
\label{eq.control_compat} y_0 ( \ell_0 ) = 0 \text{,} \qquad y_1 ( \ell_0 ) = - \max \Bigg[ \frac{ \frac{1}{2} ( y_1 ( \ell_0 ) - y_0' ( \ell_0 ) )^2 - \kappa ( \ell_0 ) }{ \tfrac{1}{2} ( y_1 ( \ell_0 ) - y_0' ( \ell_0 ) )^2 + \kappa ( \ell_0 ) }, \, 0 \Bigg] \, y_0' ( \ell_0 ) \text{.}
\end{equation}

\item Final time $T > \ell_0$, and $C^1$-admissible final data set $( \bar{\ell}_0, \bar{y}_0, \bar{y}_1 )$.
\end{enumerate}
Moreover, suppose there exists a $C^1$-admissible final branch $\mathscr{L} \in C^1 [ \bar{t}_\ast, T ]$ such that
\begin{equation}
\label{eq.control_time} \begin{cases} \ell_\ast \leq \bar{\ell}_\ast \text{,} & \ell_\ast' = \bar{\ell}_\ast' = 0 \text{,} \\ \ell_\ast < \bar{\ell}_\ast \text{,} & \text{otherwise,} \end{cases} \qquad \bar{\ell}_\ast < \bar{t}_\ast \text{.}
\end{equation}
Then, there is a control $u \in C^1 [ 0, T ]$ such that the solution $( y, \ell )$ of \crefrange{eq.debonding_y}{eq:griff_G} satisfies
\begin{equation}
\label{eq.control} \ell (T) = \bar{\ell}_0 \text{,} \qquad y ( T, \cdot ) = \bar{y}_0 \text{,} \qquad \partial_t y ( T, \cdot ) = \bar{y}_1 \text{.}
\end{equation}
\end{theorem}

\begin{remark}
Observe that when $\ell_\ast = \bar{\ell}_\ast$, the extra assumption $\ell_\ast' = \bar{\ell}_\ast' = 0$ in \cref{eq.control_time} is necessary.
In particular, when $\ell_\ast' > 0$ or $\bar{\ell}_\ast' > 0$, the continuity of $\ell' \geq 0$ implies $\ell_\ast < \bar{\ell}_\ast$.
\end{remark}

\begin{remark}
\cref{thm.control} extends to slightly more regular solutions, such as the $C^{ 1, 1 }$-solutions from \cite[Theorem 3.1]{MR3542962}.
The proof is essentially the same as for \cref{thm.control}, except one has to verify that the quantities constructed are also $C^{ 1, 1 }$.
On the other hand, the entire theory fails to extend to $C^2$-solutions---in particular, \cref{eq.griffith} shows $\ell$ can fail to be $C^2$ even when $f$ is smooth.
\end{remark}

The main ideas of the proof of \cref{thm.control} are analogous to those of \cref{thm.control_pre}.
However, here we also must ensure that derivative quantities---namely $u'$, $f'$, $\ell'$---also vary continuously.

\begin{proof}[Proof of \cref{thm.control}.]
Our assumptions and \cref{thm.wp} yield the existence of the $C^1$-solution $( \ell, y )$ of \crefrange{eq.debonding_y}{eq:griff_G}, as long as we construct $u \in C^1 [ 0, T ]$ satisfying
\begin{equation}
\label{eql.control_0} u (0) = y_0 (0) \text{,} \qquad u' (0) = y_1 (0) \text{.}
\end{equation}
Setting $u (0)$ and $u' (0)$ as in \cref{eql.control_0}, it once again suffices to define $u' \in C^0 [ 0, T ]$, as $u$ is then determined by integrating from $t = 0$.
Moreover, recall the quantities $t_\ast = \ell_\ast$ and $\ell_\ast'$ from \cref{def.initial_branch}, and recall once more that $\ell$ is independent of $u$ on $[ 0, t_\ast ]$.

\uline{Step 1.}
Our first goal is to expand the domain to the starting size $\bar{\ell}_\ast$ of the final branch.
The process for achieving this varies slightly, depending on the following cases:
\begin{enumerate}[label=(\arabic*)]
\item \label{it:c1} $\ell_\ast = \bar{\ell}_\ast$.

\item \label{it:c2} $\ell_\ast < \bar{\ell}_\ast$, which splits into additional cases:
\begin{enumerate}[label=(\alph*)]
\item \label{it:2a} $\ell_\ast' > 0$ and $\bar{\ell}_\ast' > 0$.
\item \label{it:2b} $\ell_\ast' = 0$ and $\bar{\ell}_\ast' > 0$.
\item \label{it:2c} $\ell_\ast' > 0$ and $\bar{\ell}_\ast' = 0$.
\item \label{it:2d} $\ell_\ast' = 0$ and $\bar{\ell}_\ast' = 0$.
\end{enumerate}
\end{enumerate}
In particular, we set $\ell$ and $u'$ differently in each of the above cases.

In general, we prescribe $\ell$ on $[ t_\ast, \bar{t}_\ast ]$ to be a $C^1$-function satisfying
\begin{equation}
\label{eql.control_1} ( \ell, \ell' ) ( t_\ast ) = ( \ell_\ast, \ell_\ast' ) \text{,} \qquad ( \ell, \ell' ) ( \bar{t}_\ast ) = ( \bar{\ell}_\ast, \bar{\ell}_\ast' ) \text{.}
\end{equation}
Furthermore, since by \cref{eq.control_time}, the timespan is strictly greater than the length difference,
\[
\bar{t}_\ast - t_\ast > \bar{\ell}_\ast - \ell_\ast > 0 \text{,}
\]
then we can also stipulate that $\ell$ satisfies
\begin{equation}
\label{eql.control_2} \begin{cases}
0 \leq \ell' (t) < 1 \text{,} & t \in [ t_\ast, \bar{t}_\ast ] \text{,} \\
\ell' (t) = 0 \text{,} & t \in [ t_\circ - \delta, t_\circ + \delta ] \text{.}
\end{cases} \qquad
t_\circ \coloneqq \tfrac{1}{2} ( t_\ast + \bar{t}_\ast ) \text{,} \quad 0 < \delta \ll \bar{t}_\ast - t_\ast
\end{equation}
We can impose additional conditions depending on the above cases.
\begin{itemize}
\item Case \cref{it:c1}:  \cref{eq.control_time} implies $\ell_\ast' = \bar{\ell}_\ast' = 0$.
Here, there is only possible $\ell$ that is $C^1$ and satisfies \crefrange{eql.control_1}{eql.control_2}---more specifically, we set
\begin{equation}
\label{eql.control_3a} \ell (t) = \ell_\ast = \bar{\ell}_\ast \text{,} \qquad t \in [ t_\ast, \bar{t}_\ast ] \text{.}
\end{equation}

\item Case \cref{it:c2}: Since $\ell_\ast < \bar{\ell}_\ast$, we can also set $\ell$ so that $\ell' > 0$ at most times:
\begin{align}
\label{eql.control_3b} \ell' (t) > 0 \text{,} &\qquad \begin{cases}
t \in [ t_\ast, t_\circ - \delta ) \cup ( t_\circ + \delta, \bar{t}_\ast ] \text{,} & \text{case \cref{it:2a},} \\
t \in ( t_\ast + \delta, t_\circ - \delta ) \cup ( t_\circ + \delta, \bar{t}_\ast ] \text{,} & \text{case \cref{it:2b},} \\
t \in [ t_\ast, t_\circ - \delta ) \cup ( t_\circ + \delta, \bar{t}_\ast - \delta ) \text{,} & \text{case \cref{it:2c},} \\
t \in ( t_\ast + \delta, t_\circ - \delta ) \cup ( t_\circ + \delta, \bar{t}_\ast - \delta ) \text{,} & \text{case \cref{it:2d},}
\end{cases} \\
\notag \ell' (t) = 0 \text{,} &\qquad \begin{cases}
t \in [ t_\circ - \delta, t_\circ + \delta ] \text{,} & \text{case \cref{it:2a},} \\
t \in [ t_\ast, t_\ast + \delta ] \cup [ t_\circ - \delta, t_\circ + \delta ] \text{,} & \text{case \cref{it:2b},} \\
t \in [ t_\circ - \delta, t_\circ + \delta ] \cup [ \bar{t}_\ast - \delta, \bar{t}_\ast ] \text{,} & \text{case \cref{it:2c},} \\
t \in [ t_\ast, t_\ast + \delta ] \cup [ t_\circ - \delta, t_\circ + \delta ] \cup [ \bar{t}_\ast - \delta, \bar{t}_\ast ] \text{,} & \text{case \cref{it:2d}.}
\end{cases}
\end{align}
\end{itemize}
We also define $\tau_\pm$ on $[ 0, \bar{t}_\ast ]$ using \cref{eq.tau}, with the above $\ell$.

We now claim there exists $u' \in C^0 [ 0, \tau_- ( \bar{t}_\ast ) ]$ so that \cref{eq.f} and \cref{eq.griffith} hold for $f \in C^1 [ -\ell_0, \tau_- ( \bar{t}_\ast ) ]$ and the above extension $\ell \in C^1 [ 0, \bar{t}_\ast ]$, that the compatibility conditions \cref{eql.control_0} also hold, and
\begin{equation}
\label{eql.control_3} f' ( \tau_- ( \bar{t}_\ast ) ) = \tfrac{1}{2} ( \bar{y}_1 + \bar{y}_0' ) (0) \, \tfrac{ 1 + \ell' ( \bar{t}_\ast ) }{ 1 - \ell' ( \bar{t}_\ast ) } \text{.}
\end{equation}

To prove the claim, we first set $u (0)$ and $u' (0)$ as in \cref{eql.control_0}.
Then, \cref{eq.f}, \cref{eq.ff}, and \cref{eql.control_0} imply $f$ is $C^1$ at $0$, and \cref{eq.griffith} yields that $\ell$ is $C^1$ at $t_\ast$.
In particular, $f' (0)$ is fully determined by \cref{eq.ff}.
From here, it suffices to set $f' \in C^0 [ 0, \tau_- ( \bar{t}_\ast ) ]$ so that \cref{eq.griffith} holds, since \cref{eq.ff} (via the usual iteration) then gives the desired $u' \in C^0 [ 0, \tau ( \bar{t}_\ast ) ]$.

The construction of $f'$ on $( 0, \tau_- ( \bar{t}_\ast ) ]$ varies according to each of the above-mentioned cases:
\begin{itemize}
\item For case \cref{it:c1}, we set $f' \in C^0 [ 0, \tau_- ( \bar{t}_\ast ) ]$ to be any function that satisfies \cref{eql.control_3} and
\begin{equation}
\label{eql.control_4a} 2 f' (s)^2 \leq \kappa ( \ell ( \tau_-^{-1} (s) ) ) \text{,} \qquad s \in [ 0, \tau_- ( \bar{t}_\ast ) ] \text{.}
\end{equation}
Such a $f'$ indeed exists, since the inequality in \cref{eql.control_4a} holds for $s = 0$ (by \cref{eq.griffith}, since $\ell' ( t_\ast ) = 0$) and $s = \tau_-^{-1} ( \bar{t}_\ast )$.
The latter holds since $\ell' ( \bar{t}_\ast ) = 0$, so by \cref{eq.final_branch_0} and \cref{eql.control_3},
\begin{align*}
2 f' ( \tau_- ( \bar{t}_\ast ) )^2 &= \tfrac{1}{2} | ( \bar{y}_1 + \bar{y}_0' ) (0) |^2 \\
&\leq \kappa ( \ell ( \bar{t}_\ast ) ) \text{.}
\end{align*}
Observe, from \cref{eql.control_3a} and \cref{eql.control_4a}, that this choice of $f'$ indeed satisfies \cref{eq.griffith}.

\item For case \cref{it:2a}, we see from the first part of \cref{eql.control_3b} that given any $s \in \tau_- ( [ t_\ast, t_\circ - \delta ] )$, there are up to two possible $f' (s)$ satisfying \cref{eq.griffith} (since $\ell' > 0$ in \cref{eq.griffith}).
Since $f' (0)$ is already set, we can then choose each $f' (s)$ to vary continuously in $s$.
Similarly, for $s \in \tau_- ( [ t_\circ + \delta, \bar{t}_\ast ] )$, we can choose $f' (s)$ to vary continuously in $s$ and to satisfy \cref{eq.griffith} and the boundary condition \cref{eql.control_3}.
Finally, for $s \in \tau_- ( [ t_\circ - \delta, t_\circ + \delta ] )$, we can---similarly to case \cref{it:c1}---choose any value of $f' (s)$ so that this varies continuously in $s$, connects the already determined values at $s = \tau_- ( t_\circ \pm \delta )$, and satisfies
\[
2 f' (s)^2 \leq \kappa ( \ell ( \tau_-^{-1} (s) ) ) \text{.}
\]
Once again, this gives $f' \in C^0 [ 0, \tau_- ( \bar{t}_\ast ) ]$ satisfying \cref{eq.griffith}.

\item For case \cref{it:2b}, we set $f' (s)$ for $s \in \tau_- ( [ t_\ast, t_\ast + \delta ] )$ to vary continuously in $s$, so that
\[
2 f' (s)^2 \leq \kappa ( \ell ( \tau_-^{-1} (s) ) ) \text{,} \qquad 2 f' ( \tau_- ( t_\ast + \delta ) )^2 = \kappa ( \ell ( t_\ast + \delta ) ) \text{.}
\]
In particular, $f' ( \tau_- ( t_\ast + \delta ) )$ is such that $\ell' ( t_\ast + \delta )$ is, by \cref{eq.griffith}, at the threshold of becoming positive.
Thus, for $s \in \tau_- ( ( t_\ast + \delta, t_\circ - \delta ] )$ (where $\ell' > 0$ in \cref{eq.griffith}), we can use \cref{eq.griffith} to find $f' (s)$ such that it varies continuously in $s$, even at $s = \tau_- ( t_\ast + \delta )$.
Finally, we set $f' (s)$ for $s \in \tau_- ( [ t_\circ + \delta, \bar{t}_\ast ] )$, and then for $s \in \tau_- ( [ t_\circ - \delta, t_\circ + \delta ] )$, in the same manner as in case \cref{it:2a}.

\item For case \cref{it:2c}, we first set $f' (s)$ for $s \in \tau_- ( [ t_\ast, t_\circ - \delta ] )$ as in case \cref{it:2a}.
Next, similar to case \cref{it:2b}, we set $f' (s)$ for $s \in \tau_- ( [ \bar{t}_\ast - \delta, \bar{t}_\ast ] )$ to vary continuously in $s$, so that
\[
2 f' (s)^2 \leq \kappa ( \ell ( \tau_-^{-1} (s) ) ) \text{,} \qquad 2 f' ( \tau_- ( \bar{t}_\ast - \delta ) )^2 = \kappa ( \ell ( \bar{t}_\ast - \delta ) ) \text{,}
\]
and such that \cref{eql.control_3} holds. Such values of $f' (s)$ exist by the same reasoning as in case \cref{it:c1}.
From the above, we can then use \cref{eq.griffith} to solve for continuously varying $f' (s)$ for $s \in \tau_- ( [ t_\circ + \delta, \bar{t}_\ast - \delta ] )$ (where $\ell' > 0$ by \cref{eq.griffith}).
Finally, we set $f' (s)$ for $s \in \tau_- ( [ t_\circ - \delta, t_\circ + \delta ] )$ in the same manner as in cases \cref{it:2a} and \cref{it:2b}.

\item For case \cref{it:2d}, we set $f' (s)$ for $s \in \tau_- ( [ t_\ast, t_\ast + \delta ] )$ and $s \in \tau_- ( [ t_\ast + \delta, t_\circ - \delta ] )$ as in case \cref{it:2b}.
We then set $f' (s)$ for $s \in \tau_- ( [ \bar{t}_\ast - \delta, \bar{t}_\ast ] )$ and $s \in \tau_- ( [ t_\circ + \delta, \bar{t}_\ast - \delta ] )$ as in case \cref{it:2c}.
Finally, we set $f' (s)$ for $s \in \tau_- ( [ t_\circ - \delta, t_\circ + \delta ] )$ in the same manners as in cases \cref{it:2a}--\cref{it:2c}.
\end{itemize}
Thus, in each case, we have constructed $f' \in C^0 [ 0, \tau_- ( \bar{t}_\ast ) ]$ satisfying \cref{eq.griffith} and \cref{eql.control_3}.
As mentioned before, the above also yields $u' \in C^0 [ 0, \tau_- ( \bar{t}_\ast ) ]$ saisfying \cref{eq.ff}.

Integrating $u'$ and $f'$ (the latter from $f (0) = 0$), we obtain
\[
u \in C^1 [ 0, \tau_- ( \bar{t}_\ast ) ] \text{,} \qquad f \in C^{ 0, 1 } [ -\ell_0, \tau_- ( \bar{t}_\ast ) ]
\]
satisfying \cref{eq.f} and \cref{eq.griffith}, completing the proof of the aforementioned claim.

\uline{Step 2.}
We next aim to achieve the final domain size $\bar{\ell}_0$ and half of $( \bar{y}_0, \bar{y}_1 )$.
As in the proof of \cref{thm.control_pre}, we prescribe $\ell \in C^{ 0, 1 } ( \bar{t}_\ast, T ]$ to match the given final branch:
\begin{equation}
\label{eql.control_10} \ell (t) \coloneqq \mathscr{L} (t) \text{,} \qquad t \in ( \bar{t}_\ast, T ] \text{.}
\end{equation}
As before, we extend $\tau_\pm$ to $( \bar{t}_\ast, T ]$ using \cref{eq.tau} and the above $\ell$.

Using \cref{eq.ff}, we solve for $u' \in C^0 ( \tau_- ( \bar{t}_\ast ), \tau_- (T) ]$ so that $f' \in C^0 ( \tau_- ( \bar{t}_\ast ), \tau_- (T) ]$ satisfies
\begin{equation}
\label{eql.control_11} f' (s) \coloneqq - \tfrac{1}{2} ( \bar{y}_1 + \bar{y}_0' ) ( ( \tau_+ \circ \tau_-^{-1} ) (s) - T ) \, \tfrac{ 1 + \ell' ( \tau_-^{-1} (s) ) }{ 1 - \ell' ( \tau_-^{-1} (s) ) } \text{,} \qquad s \in ( \tau_- ( \bar{t}_\ast ), \tau_- (T) ] \text{,}
\end{equation}
with $\ell$ as in \cref{eql.control_10}.
We have that \cref{eql.control_11} is well-defined, by the same reasoning as in the proof of \cref{thm.control_pre}; moreover, the right-hand side of \cref{eql.control_11} is now a continuous function of $s$.

Now, by the same computations as in the proof of \cref{thm.control_pre}, we derive that \crefrange{eql.control_10}{eql.control_11} together satisfy \cref{eq.f} and \cref{eq.griffith}, after integrating $u'$ and $f'$ as before.
Also, by \cref{eql.control_3} and \cref{eql.control_11}, we see that $f$ is $C^1$ at $\tau_- ( \bar{t}_\ast )$, and hence $u$ is also $C^1$ at $\tau_- ( \bar{t}_\ast )$.
From \cref{eq.griffith}, we have that $\ell$ is $C^1$ at $\bar{t}_\ast$, and we conclude the following quantities satisfy \cref{eq.f} and \cref{eq.griffith}:
\[
u \in C^1 [ 0, \tau_- (T) ] \text{,} \qquad f \in C^1 [ -\ell_0, \tau_- (T) ] \text{,} \qquad \ell \in C^1 [ 0, T ] \text{.}
\]

\uline{Step 3.}
The last step is to finish prescribing $( \bar{y}_0, \bar{y}_1 )$.
As in the proof of \cref{thm.control_pre}, we extend $u' \in C^0 ( \tau_- (T), T ]$ so that $f' \in C^0 ( \tau_- (T), T ]$ satisfies
\begin{equation}
\label{eql.control_12} f' (s) \coloneqq \tfrac{1}{2} ( \bar{y}_1 - \bar{y}_0' ) ( T - s ) \text{,} \qquad s \in [ \tau_- (T), T ] = [ T - \bar{\ell}_0, T ] \text{.}
\end{equation}

Letting $\alpha$ as in \cref{eq.final_set_1}, then by \cref{eq.final_branch_1}, \cref{eql.control_11}, and \cref{eql.control_12}, we have
\begin{align*}
f' ( \tau_- (T) ) &= - \tfrac{1}{2} ( \bar{y}_1 + \bar{y}_0' ) ( \bar{\ell}_0 ) \, \tfrac{ 1 + \ell' (T) }{ 1 - \ell' (T) } \text{,} \\
&= - \tfrac{1}{2} ( - \alpha \bar{y}_0' + \bar{y}_0' ) ( \bar{\ell}_0 ) \, \tfrac{ 1 + \alpha }{ 1 - \alpha } \\
&= - \tfrac{1}{2} ( 1 + \alpha ) \, \bar{y}_0 ( \bar{\ell}_0 ) \text{,} \\
\lim_{ s \searrow \tau_- (T) } f' (s) &= \tfrac{1}{2} ( \bar{y}_1 - \bar{y}_0' ) ( \bar{\ell}_0 ) \\
&= - \tfrac{1}{2} ( 1 + \alpha ) \, \bar{y}_0 ( \bar{\ell}_0 ) \text{,}
\end{align*}
and it follows that $f'$ and $u'$ are continuous at $\tau_- (T)$.
Integrating $u'$ and $f'$ as before, we have
\[
u \in C^1 [ 0, T ] \text{,} \qquad f \in C^1 [ -\ell_0, T ] \text{,} \qquad \ell \in C^1 [ 0, T ]
\]
that form a solution of \cref{eq.f} and \cref{eq.griffith}.

From the above solution $( \ell, f )$, we obtain from \cref{eq.y} a $C^1$-solution $( \ell, y )$ of \crefrange{eq.debonding_y}{eq:griff_G}.
The same computations as in the proof of \cref{thm.control_pre} (using \cref{eq.compat_0}, \cref{eq.yyy}, \cref{eq.final_set_0}, \cref{eql.control_11}, and \cref{eql.control_12}) yield that \cref{eq.control} holds, completing the proof of the theorem.
\end{proof}

\begin{remark}
In Step 1 of the proof of \cref{thm.control}, we imposed a middle region $[ t_\circ - \delta, t_\circ + \delta ]$ where $\ell' = 0$.
This is because $f' ( \tau_- (t) )$ may have opposite signs near $t = t_\ast$ and $t = \bar{t}_\ast$, and this middle region provides a way for $f'$ to change signs continuously.
This middle region is not strictly needed in cases \crefrange{it:2b}{it:2d}, as we could also choose the sign of $f'$ near either $t_\ast$ or $\bar{t}_\ast$.
\end{remark}

Again, for the special case of static final branches, we can reduce \cref{thm.control} to a version for which we need not refer to the technical admissible final branch condition.

\begin{corollary}[$C^1$-controllability, static final branch] \label{thm.control_static}
Assume the following conditions hold:
\begin{enumerate}[label=(\roman*)]
\item Initial data $\ell_0 > 0$ and $( y_0, y_1 ) \in ( C^1 \times C^0 ) [ 0, \ell_0 ]$ satisfying
\begin{equation}
\label{eq.control_static_compat} y_0 ( \ell_0 ) = 0 \text{,} \qquad y_1 ( \ell_0 ) = - \max \Bigg[ \frac{ \frac{1}{2} ( y_1 ( \ell_0 ) - y_0' ( \ell_0 ) )^2 - \kappa ( \ell_0 ) }{ \tfrac{1}{2} ( y_1 ( \ell_0 ) - y_0' ( \ell_0 ) )^2 + \kappa ( \ell_0 ) }, \, 0 \Bigg] \, y_0' ( \ell_0 ) \text{.}
\end{equation}

\item Final time $T > 0$ and final data $\bar{\ell}_0 > 0$, $( \bar{y}_0, \bar{y}_1 ) \in ( C^1 \times C^0 ) [ 0, \bar{\ell}_0 ]$ satisfying
\begin{equation}
\label{eq.control_static_time} \begin{cases} \bar{\ell}_0 \geq \ell_\ast \text{,} & \ell_\ast' = 0 \text{,} \\ \bar{\ell}_0 > \ell_\ast \text{,} & \text{otherwise,} \end{cases} \qquad T > 2 \bar{\ell}_0 \text{,}
\end{equation}
as well as the compatibility and constraint conditions
\begin{equation}
\label{eq.control_static_const} \bar{y}_0 ( \bar{\ell}_0 ) = \bar{y}_1 ( \bar{\ell}_0 ) = 0 \text{,} \qquad \| \bar{y}_1 + \bar{y}_0' \|_{ L^\infty [ 0, \bar{\ell}_0 ] }^2 \leq 2 \, \kappa ( \bar{\ell}_0 ) \text{.}
\end{equation}
\end{enumerate}
Then, there exists a control $u \in C^1 [ 0, T ]$ so that the solution $( y, \ell )$ of \crefrange{eq.debonding_y}{eq:griff_G} satisfies
\begin{equation}
\label{eq.control_static} \ell (T) = \bar{\ell}_0 \text{,} \qquad y ( T, \cdot ) = \bar{y}_0 \text{,} \qquad \partial_t y ( T, \cdot ) = \bar{y}_1 \text{.}
\end{equation}
\end{corollary}

\begin{proof}
From \cref{eq.control_static_const}, we have that the function $\mathscr{L}$ given by \cref{eq.final_branch_static} is a $C^1$-admissible final branch, with $\bar{t}_\ast = T - \bar{\ell}_0$, $\bar{\ell}_\ast = \bar{\ell}_0$, and $\bar{\ell}_\ast' = 0$.
The above and \cref{eq.control_static_time} imply \cref{eq.control_time} holds with $\mathscr{L}$ as above.
Applying \cref{thm.control} with the above $\mathscr{L}$ results in the desired control.
\end{proof}

\vspace{0.2cm}

\section*{Acknowledgments}

N.~De Nitti is a member of the Gruppo Nazionale per l'Analisi Matematica, la Probabilità e le loro Applicazioni (GNAMPA) of the Istituto Nazionale di Alta Matematica (INdAM).

A.~Shao was supported, for part of this work, by EPSRC grant EP/Y021487/1.

We thank the Centro de Ciencias de Benasque Pedro Pascual, where part of this work was carried out during the Benasque X Workshop--Summer School 2024 on \textit{Partial Differential Equations, Optimal Design and Numerics}. 
N.~De Nitti also gratefully acknowledges the hospitality of Queen Mary University of London, where part of this work was also completed. 

We are grateful to F.~Riva and E.~Zuazua for helpful discussions related to this work.

\vspace{0.5cm}

\printbibliography

\vfill 

\end{document}